\documentclass[10pt,letterpaper]{amsart}
\usepackage[utf8]{inputenc}\usepackage[left=2.8cm,right=2.8cm,top=3.6cm,bottom=3cm]{geometry}
\usepackage{amssymb,latexsym,amsmath,amsthm,amscd}
\usepackage{graphicx}
\usepackage{leftidx}
\usepackage{dsfont}
\usepackage{stmaryrd}
\usepackage{bm}
\usepackage{comment}
\usepackage[pagebackref=false,colorlinks]{hyperref}
\hypersetup{pdffitwindow=true,linkcolor=blue,citecolor=olive,urlcolor=cyan}
\usepackage{hyperref}
\usepackage[all]{xy}
\usepackage{mathrsfs}
\usepackage{color}
\usepackage[dvipsnames]{xcolor}
\definecolor{Blue}{rgb}{0.3,0.3,0.9}

\usepackage[T2A,OT1]{fontenc}
\DeclareSymbolFont{cyrillic}{T2A}{cmr}{m}{n}
\DeclareMathSymbol{\Sha}{\mathalpha}{cyrillic}{216}

\theoremstyle{plain}
\newtheorem{thm}{Theorem}[subsection] % reset theorem numbering for each chapter

\usepackage{mathrsfs}

\theoremstyle{definition}

\newtheorem{defn}[thm]{Definition}

\newtheorem{rmk}[thm]{Remark}
\newtheorem*{rmk*}{Remark}

\newtheorem{assumption}{Assumption}[section]

\theoremstyle{definition}

\theoremstyle{plain}
\newtheorem{prop}[thm]{Proposition}
\theoremstyle{plain}
\newtheorem{lem}[thm]{Lemma}
\theoremstyle{plain}
\newtheorem{cor}[thm]{Corollary}
\newtheorem*{thm-intro}{Theorem}
\newtheorem*{cor-intro}{Corollary}

\theoremstyle{remark}

\newtheorem{rem}[thm]{Remark}

\numberwithin{equation}{section}

\newcommand{\cc}{\mathbf{c}}

\newcommand{\ch}{\xi}

\newcommand{\rH}{{\mathrm{H}}}

\newcommand{\pp}{\mathfrak{p}}
\newcommand{\ppbar}{\bar{\mathfrak{p}}}
\newcommand{\fkm}{\mathfrak{m}}
\newcommand{\fkmbar}{\overline{\mathfrak{m}}}
\newcommand{\fkf}{\mathfrak{f}}

\newcommand{\fkl}{\mathfrak{l}}

\newcommand{\fkn}{\mathfrak{n}}

\newcommand{\fa}{\mathfrak{a}}

\newcommand{\relstr}{{{\rm rel},{\rm str}}}
\newcommand{\strrel}{{{\rm str},{\rm rel}}}
\newcommand{\ord}{{{\rm ord},{\rm ord}}}

\DeclareMathOperator{\GL}{GL}

\DeclareMathOperator{\pa}{\mathfrak{p}}

\DeclareMathOperator{\et}{\acute{e}t}

\DeclareMathOperator{\Ll}{\mathscr{L}}
\DeclareMathOperator{\Ss}{\mathscr{S}}
\DeclareMathOperator{\p}{\mathfrak{P}}

\DeclareMathOperator{\tpr}{\text{pr}}
\DeclareMathOperator{\La}{\Lambda}
\DeclareMathOperator{\la}{\lambda}
\DeclareMathOperator{\Ga}{\Gamma}

\DeclareMathAlphabet\mathbfcal{OMS}{cmsy}{b}{n}
\DeclareTextSymbolDefault{\uhorn}{T5}
\DeclareTextSymbolDefault{\ocircumflex}{T5}
\DeclareTextSymbolDefault{\acircumflex}{T5}

\newcommand{\bC}{\mathbf{C}}

\def\cA{{\mathcal A}}  %automorphic forms

\def\cL{{\mathcal L}}

\def\cM{\mathcal M}
\def\cO{\mathcal O}

\def\cP{{\mathcal P}}

\newcommand{\VQdag}{{\mathbf{V}_{\underline{Q}}^\dagger}}
\newcommand{\Vsdag}{{\mathbf{V}_{Q_1}^\dagger}}
\newcommand{\Vdag}{{\mathbf{V}^\dagger}}
\newcommand{\Adag}{{\mathbf{A}^\dagger}}

\newcommand{\unb}{{\boldsymbol{f}}}

\newcommand{\Q}{\mathbf{Q}}
\newcommand{\Z}{\mathbf{Z}}

\def\makeop#1{\expandafter\def\csname#1\endcsname
	{\mathop{\rm #1}\nolimits}\ignorespaces}
\makeop{Hom}   \makeop{End}   \makeop{Aut}   %\makeop{Isom}
\makeop{Pic} \makeop{Gal}       \makeop{Div} \makeop{Lie}
\makeop{PGL}   \makeop{Corr} \makeop{PSL} \makeop{sgn} \makeop{Spf}
\makeop{Tr} \makeop{Nr} \makeop{Fr} \makeop{disc}
\makeop{Proj} \makeop{supp} \makeop{ker}   \makeop{Im} \makeop{dom}
\makeop{coker} \makeop{Stab} \makeop{SO} \makeop{SL} \makeop{SL}
\makeop{Cl}    \makeop{cond} \makeop{Br} \makeop{inv} \makeop{rank}
\makeop{id}    \makeop{Fil} \makeop{Frac}  \makeop{GL} \makeop{SU}
\makeop{Trd}   \makeop{Sp} \makeop{Tr}    \makeop{Trd} \makeop{Res}
\makeop{ind} \makeop{depth} \makeop{Tr} \makeop{st} \makeop{Ad}
\makeop{Int} \makeop{tr}    \makeop{Sym} \makeop{can} \makeop{SO}
\makeop{torsion} \makeop{GSp} \makeop{Tor}\makeop{Ker} \makeop{rec}
\makeop{Ind} \makeop{Coker}
\makeop{vol} \makeop{Ext} \makeop{gr} \makeop{ad}
\makeop{Gr}\makeop{corank} \makeop{Ann}
\makeop{Hol} %Holomorphic
\makeop{Fitt} \makeop{Mp} \makeop{CAP}

\newcommand{\dBr}[1]{\llbracket{#1}\rrbracket}

\newcommand{\bT}{\mathbb{T}}

\newcommand{\cR}{\mathbb{I}}

\newcommand{\bfff}{{\boldsymbol{f}}}
\newcommand{\bff}{{\boldsymbol{f}}}
\newcommand{\bfg}{{\boldsymbol{g}}}
\newcommand{\bfh}{{\boldsymbol{h}}}

\begin{document}
	
\title{Anticyclotomic Euler system over biquadratic fields}
\author[K.\,T.\,Do]{Kim Tuan Do}
	
%\keywords{Elliptic curves, Birch and Swinnerton-Dyer conjecture, $p$-adic families of modular forms, $p$-adic $L$-functions, Euler systems}	
\date{\today} %, {\color{blue} draft in progress}}

\address[]{Department of Mathematics, University of California Los Angeles, CA 90095, USA}
\address[]{Department of Mathematics, University of California Santa Barbara, CA 93106, USA}
\email{ktdo@ucsb.edu}

%\commby{}

% ----------------------------------------------------------------

\begin{abstract}
We construct a new Euler system (anticyclotomic, in the sense of Jetchev--Nekov\'a{\v r}--Skinner) for the Galois representation $V_{f,\chi}$ attached to a newform $f$ of weight $k\geq 2$ twisted by an anticyclotomic Hecke character $\chi$ defined over an imaginary biquadratic field $K_0$. We then show some arithmetic applications of the constructed Euler system, including results on the Bloch--Kato conjecture and a divisibility towards the Iwasawa--Greenberg main conjecture for $V_{f,\chi}$.

\end{abstract}

% ----------------------------------------------------------------

\maketitle
\setcounter{tocdepth}{2}
\tableofcontents

\addtocontents{toc}{\setcounter{tocdepth}{2}}
\newpage
\section{Introduction}

\addtocontents{toc}{\setcounter{tocdepth}{-10}}

{
\renewcommand{\thethm}{\Alph{thm}}

Let $f=\sum_{n=1}^\infty a_nq^n\in S_{k}(\Gamma_0(N_f))$ be an elliptic newform of even weight $k=2r\geq 2$,  and let $p\nmid 6N_f$ be a prime. Let $K_0/\Q$ be an imaginary biquadratic field in which $p$ splits. This means that $K_0$ contains two distinct imaginary quadratic subfields $K_1$, $K_2$ together with one real quadratic subfield $K_3$.  Let $L$ be a number field containing $K_0$ and the Fourier coefficients of $f$, and let $\mathfrak{P}$ be a prime of $L$ above $p$ at which $f$ is ordinary, i.e. $v_{\mathfrak{P}}(a_p)=0$. Let $\chi$ be an anticyclotomic Hecke character of $K_0$ with infinity type $(-a,a,-b,b)$ where $a\ge b\ge 0$ \footnote{By either using $L(f/K_0,\chi,r)=L(f/K_0,\chi^\cc,r)$, where $\chi^\cc$ is the composition of $\chi$ with the action of complex conjugation, or swapping the order of $K_1$ and $K_2$, we would be able to cover other cases of $a$ and $b$.} that satisfies the decomposition hypothesis (\ref{eq:chi-decomp}) i.e. $\chi$ can be factored
 \begin{equation*}
\chi=\tilde{\psi}_1\tilde{\psi}_2\mathbf{N}^{(k_1+k_2-2)/2}.
    \end{equation*}
Here, for $i\in\{1,2\}$, $\psi_i$ is a Hecke character of $K_i$ of infinity type $(1-k_i,0)$ and modulus  $\fkf_i$; $\tilde{\psi}_i$ is the Hecke character of $K_0$, obtained by composing $\mathbb{A}_{K_0}^{\times}\xrightarrow{\mathbb{N}_{K_0/K_i}}\mathbb{A}_{K_i}^{\times}\xrightarrow{\psi_i}\mathbb{C}$. Not that if this happens, we must have $k_1=a-b+1$ and $k_2=a+b+1$. We then focus on the conjugate self-dual $G_{K_0}={\rm Gal}(\overline{\Q}/K_0)$-representation
\[
V_{f,\chi}:=V_f^\vee(1-r)\otimes\chi^{-1},
\] 
where $V_f^\vee$ is the contragredient of Deligne's $\mathfrak{P}$-adic Galois representation associated to $f$.

Throughout the remainder of this section, 
we assume the following hypotheses:
\begin{itemize}
    \item $f$ is ordinary and non-Eisenstein at $\mathfrak{P}$;
    \item $p$ splits completely in $K_0$;
    \item $p\nmid h_{K_0}$, where $h_{K_0}$ is the class number of $K_0$.
\end{itemize}
For every integral ideal $\mu_3$ of $\cO_{K_3}$, let $K_0[\mu_3]$ be the maximal $p$-subextension of the ring class field of $K_0$ of conductor $\mu_3$. Denote by $\mathcal{N}$ the set of squarefree products of primes $\mu_3\subset\cO_{K_3}$, where $m=N_{K_3/\Q}(\mu_3)$ is squarefree, prime to $p$, and split in $K_0$.

\begin{thm}[Theorem\,\ref{maintheorem2}]
\label{thmA} 
There exists a collection of Iwasawa cohomology classes
\[
\mathbf{z}_{f,\chi,\mu_3}\in H^1_{\rm Iw}\bigl(K_0[\mu_3 p^\infty],T_{f,\chi}\bigr),
\] 
indexed by the ideals $\mu_3\in\mathcal{N}$ with $m=N_{K_3/\Q}(\mu_3)$, where $T_{f,\chi}$ is a certain $G_K$-stable $\mathcal{O}$-lattice inside $V_{f,\chi}$,
such that for every prime $\la_3\in \mathcal{N}$ of norm $\ell$, with $(\ell,mp)=1$ we have the norm relation
\begin{equation*}
    \mathrm{Norm}_{K_0[\mu_3]}^{K_0[\mu_3\la_3]}(\mathbf{z}_{f,\chi,\mu_3\la_3})=P_{\cL_4}(\mathrm{Frob}_{\cL_4})(\mathbf{z}_{f,\chi,\mu_3}),
\end{equation*}
where $P_{\cL_4}(X)=\det(1-X\cdot \mathrm{Frob}_{\cL_4}\,|\,(T_{f,\chi})^\vee(1))$, and ${\rm Frob}_{\cL_4}$ is the geometric Frobenius.
\end{thm}

\begin{rmk*}
In \cite{JNS}, Jetchev--Nekov{\'a}{\v{r}}--Skinner have developed a theory of `split' anticyclotomic Euler systems attached to conjugate self-dual representations over CM fields, where classes are defined over ring class extensions of CM fields (indexed by ideals of their totally real subfields). Our construction fits within their framework. Furthermore, we note that the condition where $m=N_{K_3/\Q}(\mu_3)$ splits in $K_0$ does exclude the setting when $m$ is inert in $K_3$ and $\mu_3$ splits in $K_0$. Nevertheless, this does not affect the application of the \cite{JNS} machinery (see some details for the imaginary quadratic case in \cite[\S{4.3}]{Do-PhD}).
\end{rmk*}

Due to its geometric origin, if we let
\[
\kappa_{f,\chi}:={\rm Norm}^{K_0[1]}_{K_0}(\mathbf{z}_{f,\chi,(1)})
\]
then it will land in a Selmer subgroup of $H^1(K_0,V_{f,\chi})$ with `nice' local conditions (see Section \ref{subsec:local-condition}). Then feeding Theorem~\ref{thmA} to the general Euler system machinery of \cite{JNS}, we deduce the following cases of the Bloch--Kato conjecture in analytic rank $0$.

\begin{thm}[Theorem\,\ref{thm:BK-def}] %Theorem\,\ref{thm:BK-indef}]
\label{thmB}
Let $f\in S_k(\Gamma_0(N_f))$ be a newform. Let $\chi$ be an anticyclotomic Hecke character of $K_0$ of infinity type $(-a,a,-b,b)$ satisfying the Hypotheses (\ref{eq:chi-decomp}). Assume further that:
\begin{enumerate}
    \item Either $k\ge 2a+2$ or $2b\ge k$; 
    \item $N_f\cO_{K_3}=\mathfrak{n}^+\mathfrak{n}^-$ where $\fkn^+$
(respectively $\fkn^-$) is divisible only by primes which are split (respectively inert) in $K_0/K_3$ and $\fkn^-$ is a squarefree product of an even number of primes;
    %\item $f$ is not of CM type;
    \item $\bar{\rho}_f$ is absolutely irreducible; 
    \item $(pN_f,\mathrm{Norm}_{K_1/\Q}(\mathfrak{f}_1)\mathrm{Norm}_{K_2/\Q}(\mathfrak{f}_2)D_{K_0})=1$;
\end{enumerate}
Then
\[
L(f/K_0,\chi,r)\neq 0\quad\Longrightarrow\quad{\rm Sel}_{\rm BK}(K_0,V_{f,\chi})=0,
\] 
and hence the Bloch--Kato conjecture for $V_{f,\chi}$ holds in this case.
\end{thm}

Note that the first $2$ conditions of Theorem \ref{thmB} imply that the sign of the functional equation of $V_{f,\chi}$ is equal to $+1$, see also Remark \ref{rootnumber}. This puts us in an ideal situation for the non-vanishing of central $L-$values generically. 

Let $\cO$ be the ring of integers of $L_\mathfrak{P}$. We say that $f$ has {big image} if for a certain Galois stable $\cO$-lattice $T_f^\vee\subset V_f^\vee$, the image of $G_\Q$ in ${\rm Aut}_\cO(T_f^\vee)$ contains a conjugate of ${\rm SL}_2(\Z_p)$. Under this assumption, we also have results towards the Bloch--Kato conjecture in the analytic rank $1$ case.
\begin{thm}[Theorem\,\ref{thm:BK-def-1}] 
\label{thmC}

Let the hypotheses be as in Theorem~\ref{thmB},  and assume in addition that: 
\begin{enumerate}
    \item $\bar{\rho}_f$ is $p$-distinguished;
    \item $f$ has big image;
    \item $p>k-2$.
\end{enumerate}
If $2a\geq k\ge 2b+2$ (which implies $L(f/K,\chi,r)=0$), then 
\[
{\rm dim}_{L_\mathfrak{P}}\,{\rm Sel}_{\rm BK}(K_0,V_{f,\chi})\geq 1.
\]

%Moreover, there exists a class $z_{f,\chi}\in{\rm Sel}_{\rm BK}(K_0,V_{f,\chi})$ such that
%\[
%z_{f,\chi}\neq 0\quad\Longrightarrow\quad{\rm dim}_{L_\mathfrak{P}}\,{\rm Sel}_{\rm BK}(K_0,V_{f,\chi})=1.
%\]
\end{thm}

Finally, we note that these results also include the proof of a divisibility towards the anticyclotomic Iwasawa Main Conjecture for $V_{f,\chi}$, see Theorem\,\ref{thm:IMC-def}.

\subsection{Relation to previous works}\label{subsec:previous}

When $\chi$ is an anticyclotomic Hecke character over $K$, an imaginary quadratic field, the arithmetic of $V_{f,\chi}$ has been studied intensively via the Euler system of Heegner points pioneered by Gross--Zagier and Kolyvagin \cite{GZ,kol88} (see also \cite{zhang130,Tian-PhD,nekovar-CM}), and generalized Heegener cycles by Bertolini--Darmon--Prasanna \cite{bdp1}. In particular, these objects have direct implications towards the Bloch-Kato conjecture in analytic rank $0$ for $V_{f,\chi}$  by either varying the generalised Heegner cycles in $p$-adic families like in Castella--Hsieh \cite{cas-hsieh1} (see also \cite{cas-2var}), or by the `level-raising' method like in Bertolini--Darmon \cite{bdIMC} (see also \cite{LV-JNT,ChHs2,Chi}). 
In the same vein as \cite{bdIMC}, Nekov{\'a}{\v{r}} \cite{Neko_2012} and Wang \cite{Wang-anticycltomicHilbert} proved results towards the rank $0$ Bloch-Kato conjecture when $f$ is a cuspidal Hilbert modular eigenform over a totally real field $F$ of parallel weight $2$ and higher weights respectively, where $\chi$ is a finite order character, see also result of Tamiozzo \cite{Tamiozzo2021}. 

Outside of the Heegner realm, it is worthwhile to mention that the Euler system of Beilinson--Flach classes  constructed by Lei--Loeffler--Zerbes \cite{LLZ,LLZ-K} and Kings--Loeffler--Zerbes \cite{KLZ,KLZ-AJM} can be applied to obtain similar rank $0$ results. Relying on this, Lamplugh \cite{JackLamplugh} constructed Euler systems for ${\rm Ind}_{K_0}^{K_1}\mathcal{O}(\chi\rho)$ over $K_1$ (where $\rho$ is an auxiliary character) and used that to bound the associated Selmer group over the $K_0$ via Rubin's machinery \cite{Rubin-ES}.

The anticyclotomic Euler system over $K_0$ that we will describe in this paper is more comparable with the anticyclotomic diagonal Euler system \cite{Do-PhD,CD} over $K$ (an imaginary quadratic field) and comes together with application towards the Bloch-Kato conjecture in analytic rank $0$. The construction of the cohomology classes, similar to \cite{Do-PhD,CD}, is based on a generalisation of the diagonal cycles pioneered by Gross--Kudla \cite{gross-kudla} and Gross--Schoen \cite{gross-schoen}, and improved recently by Darmon--Rotger and Bertolini--Seveso--Venerucci (see \cite{BDRSV}). Despite the fact that it is being done later, the imaginary biquadratic case is actually a generic case (where $K_1\neq K_2$) while the imaginary quadratic case is a degenerate situation (where $K=K_1=K_2$).

In future work, we intend to construct a bipartite Euler system over a biquadratic field as well as investigate the case where $p$ does not split completely in $K_0$.

\subsection{Acknowledgements} 
The author would like to thank Francesc Castella and Christopher Skinner for many fruitful discussions and encouragements, Haruzo Hida for suggesting the problem, Chandrashekhar Khare, Romyar Sharifi and Alex Smith for helpful communication, and the anonymous referee for many helpful comments and suggestions on an earlier
draft of this paper.

\addtocontents{toc}{\setcounter{tocdepth}{2}}

}

\section{Preliminaries}

\subsection{Galois representations associated to newforms} In this section, we follow \cite[Sec.\,1.1]{CD} and introduce some important notation and results. Let $f\in S_k(\Gamma_1(N_f),\chi_f)$ be a normalized newform of weight $k\ge 2$ and let $\sum_{n=1}^\infty a_n q^n$ be its $q$-expansion. Let $p\nmid N_f$ be a prime. Fix embeddings $i_\infty: \overline{\Q}\hookrightarrow \bC$ and $i_p: \overline{\Q}\hookrightarrow \overline{\Q}_p$. Let $L/\Q$ be the coefficient field of $f$ that is, $L$ contains all values $i_\infty^{-1}(a_n)$ and $i_\infty^{-1}\circ \chi_f$. Let $\mathfrak{P}$ be the prime of $L$ above $p$ with respect to $i_p$. Let $S=\{\text{prime }\ell| pN_f\}\cup \{\infty\}$. Then Eichler-Shimura (for $k=2$) and Deligne (for $k>2$) construct a $p-$adic Galois representation associated to $f$:
\begin{equation*}
    \rho_{f,\mathfrak{P}}:G_{\Q,S}\rightarrow \GL_2(L_{\mathfrak{P}}),
\end{equation*}
such that for all primes $\ell\notin S$:
\begin{itemize}
    \item trace$(\rho_{f,\mathfrak{P}}(\mathrm{Frob}_\ell))=i_p(a_\ell)$,
    \item det$(\rho_{f,\mathfrak{P}}(\mathrm{Frob}_\ell))=i_p(\chi_f(\ell)\ell^{k-1})$,
    \item $\rho_{f,\mathfrak{P}}$ is irreducible, hence absolutely irreducible.
\end{itemize}
 Here $\mathrm{Frob}_\ell$ is the geometric Frobenius.

 As in \cite[Sec.\,1.1]{CD}, one obtains the geometric realization $V_f$ of $\rho_{f,\mathfrak{P}}$ defined as the subspace of
\[
H^1_{\et}(Y_1(N_f)_{\overline{\Q}},\Ss_{k-2})\otimes L_{\mathfrak{P}}.
\]
Dually, $V_f^{\vee}={\rm Hom}(V_f,L_{\mathfrak{P}})$ 
 can be interpreted as the maximal quotient of 
\[
H^1_{\et}(Y_1(N_f)_{\overline{\Q}},\Ll_{k-2}(1))\otimes L_{\mathfrak{P}}
\]
on which the dual Hecke operator $T_\ell'$ acts as multiplication by $a_\ell$ for all $\ell\nmid N_fp$ and $\langle d\rangle = \langle d\rangle^{*}$ acts as multiplication by $\chi_f(d)$ for all $d\in (\Z/N_f\Z)^\times$. 

Let $\cO$ be the ring of integers of $L_\mathfrak{P}$. There exists a $G_{\Q}$-stable $\cO$-lattice $T_f^\vee\subset V_f^\vee$ defined as the image of 
$H^1_{\et}(Y_1(N_f)_{\overline{\Q}},\Ll_{k-2}(1))\otimes\cO$ in $V_f^\vee$.

If $f$ is ordinary at $p$ (which means $i_p(a_p)\in\cO^\times$), then the restriction of $V_f$ to $G_{\Q_p}$ is reducible. This leads us to an exact sequence of $L_{\mathfrak{P}}[G_{\Q_p}]$-modules
\begin{equation*}
0\rightarrow V_f^{+}\rightarrow V_f\rightarrow V_f^{-}\rightarrow 0, 
\end{equation*}
where ${\rm dim}_{L_\mathfrak{P}}V_f^{\pm}=1$. Dually, we also obtain an exact sequence for the restriction of $V_f^{\vee}$ to $G_{\Q_p}$
\begin{equation}\label{subsec:p-ord}
0\rightarrow V_f^{\vee,+}\rightarrow V_f^{\vee}\rightarrow V_f^{\vee,-}\rightarrow 0, 
\end{equation}
where $V_f^{\vee,+}\simeq(V_f^-)^\vee(1-k)(\chi_f^{-1})$, and the $G_{\Q_p}$-action on the quotient $V_{f}^{\vee,-}$ is given by the unramified character sending the arithmetic Frobenius ${\rm Frob}_p^{-1}$ to $\alpha_p$, which is the unit root of $x^2-a_px+\chi_f(p)p^{k-1}$.

\subsection{Patched CM Hecke modules}\label{LLZ-K}

Here, we recall the conventions on Hecke characters and the construction of certain patched CM Hecke modules from \cite[Sec.\,1.3]{CD} and \cite{LLZ-K}.

\subsubsection{Hecke characters and theta series }\label{theta}

Let $K$ be an imaginary quadratic field. Let $p=\pp\ppbar$ be a prime that splits in $K$ with $\pp$, the prime of $K$ above $p$, induced by $i_p:\overline{\Q}\hookrightarrow\overline{\Q}_p$. We say that a Hecke character $\psi:\mathbb{A}_K^{\times}/K^{\times}\rightarrow \mathbb{C}^{\times}$ has infinity type $(m,n)$, where $m,n$ are integers, if $\psi_{\infty}(x_{\infty})=x_{\infty}^m\bar{x}_{\infty}^n$. 

Let $\text{rec}_K:\mathbb{A}_K^{\times}\rightarrow G_K^{\text{ab}}$ be the geometrically normalized Artin reciprocity map. Following \cite[Sec.\,1.3.1]{CD}, given $g\in G_K$, we take $x\in\mathbb{A}_K^{\times}$ such that $\text{rec}_K(x)=g\vert_{K^{\rm ab}}$ and define
\begin{equation}
        \psi_{\p}(g)=i_p\circ i_{\infty}^{-1}(\psi(x)x_{\infty}^{-m}\bar{x}_{\infty}^{-n})x_{\pa}^m x_{\bar{\pa}}^n.\nonumber
\end{equation}
Such a $\psi_{\p}$ will be called the $p$-adic avatar of $\psi$. We shall also use $\psi$ to denote its $p$-adic avatar if the context makes this usage reasonable.

Attached to $\psi$, a Hecke character of $K$ of infinity type $(-1,0)$ with conductor $\fkf$ that takes values in a finite extension $L/K$, is the theta series 
\begin{equation*}
\theta_{\psi}=\sum_{(\mathfrak{a},\fkf)=1}\psi(\mathfrak{a})q^{N_{K/\Q}(\mathfrak{a})} \in S_2(\Gamma_1(N_\psi),\chi_\psi\epsilon_K)
\end{equation*}
where $N_\psi=N_{K/\Q}(\fkf){\rm disc}(K/\Q)$, $\chi_\psi$ is the unique Dirichlet character modulo $N_{K/\Q}(\fkf)$ such that 
$\psi((n))=n\chi_\psi(n)$ for all $n\in\Z$ with $(n,N_{K/\Q}(\fkf))=1$, and $\epsilon_K$ is the quadratic Dirichlet character attached to $K$. 
The cuspform $\theta_\psi$ is new of level $N_\psi=N_{K/\Q}(\mathfrak{f})\cdot \text{disc}(K/\Q)$ by \cite{Mi}. 
One obtains the following description of the $\mathfrak{P}$-adic representation of $\theta_\psi$
\[
\quad V_{\theta_{\psi}}^{\vee}\cong {\rm Ind}^{\Q}_KL_{\mathfrak{P}}(\psi^{-1}),
\]
where $\mathfrak{P}$ is the prime of $L$ above $p$ with respect to $i_p$.
\subsubsection{Hecke algebras and norm maps}\label{LLZrecall}

We keep the notation of the previous section and follow \cite[Sec.\,1.3.1]{CD}. Let $\fkn\subset\cO_K$ be an ideal divisible by $\fkf$ and let $N=N_{K/\Q}(\fkn){\rm disc}(K/\Q)$. Let $K_{\fkn}$ be the ray class field of $K$ with conductor $\fkn$. Let $H_\fkn={\rm Gal}(K_{\fkn}/K)$ be the ray class group of $K$ modulo $\fkn$. Let $K(\fkn)$ be the largest $p$-subextension of $K$ contained in $K_{\fkn}$, i.e. ${\rm Gal}(K(\fkn)/K)\cong H_{\fkn}^{(p)}$ is the largest $p$-power quotient of $H_{\fkn}$. Given an ideal $\mathfrak{k}$ of $K$ that is coprime to $\fkn$, let $[\mathfrak{k}]$ be the class of $\mathfrak{k}$ in $H_{\fkn}$. Let $\mathbb{T}'(N)$ be the subalgebra of ${\rm End}_{\Z}(H^1(Y_1(N)(\bC),\Z))$ generated by $\langle d\rangle'$ and $T_\ell'$ for all primes $\ell$, then one can prove that:

\begin{prop}[Proposition 3.2.1 in \cite{LLZ-K}]\label{phi_n}
 There exists a homomorphism $\phi_{\fkn}: \mathbb{T}'(N)\rightarrow \mathcal{O}[H_{\fkn}]$ 
defined by 
    \begin{align*}
        \phi_{\fkn}(T_\ell')&=\sum_{\substack{\mathfrak{l}\subset\cO_K,\fkl\nmid\mathfrak{n},\\N_{K/\Q}(\fkl)=\ell}}[\fkl]\psi(\fkl),\\
        \phi_{\fkn}(\langle d\rangle ')&=\chi_\psi(d)\epsilon_{K}(d)[(d)].
    \end{align*}
\end{prop}

\vspace{10pt}

For $\fkm=\fkn\fkl$, with $\fkl$ a prime ideal and $(\fkm,p)=1$, put $M=N_{K/\Q}(\fkm)\text{disc}(K/\Q)$ and one has the following map
\begin{equation}
\mathcal{N}_{\fkn}^{\fkm}:\mathcal{O}[H_{\fkm}^{(p)}]\otimes_{\mathbb{T}'(M)\otimes\Z_p,\phi_{\fkm}}H^1_{\et}(Y_1(M)_{\overline{\Q}},\Z_p(1))\rightarrow \mathcal{O}[H_{\fkn}^{(p)}]\otimes_{\mathbb{T}'(N)\otimes\Z_p,\phi_{\fkn}}H^1_{\et}(Y_1(N)_{\overline{\Q}},\Z_p(1)).\nonumber
\end{equation}
This norm map is defined explicitly by splitting into $3$ cases (see \cite[Sec.\,1.1.2]{CD} for the definition of the degeneracy map): 
\begin{itemize}
    \item If $\fkl\mid\fkn$ then 
    \begin{equation*}
        \mathcal{N}_{\fkn}^{\fkm}=1\otimes \text{pr}_{1*};
    \end{equation*}
    \item If $\fkl\nmid\fkn$ is split or ramified in $K$ and $N_{K/\Q}(\fkl)=\ell$, then 
    \begin{equation*}
        \mathcal{N}_{\fkn}^{\fkm}=1\otimes \text{pr}_{1*}-\frac{\psi(\mathfrak{l)}[\fkl]}{\ell}\otimes \text{pr}_{\ell*};
    \end{equation*}
    \item If $\fkl\nmid\fkn$ is inert in $K$, say $\fkl=(\ell)$, then 
    \begin{equation*}
        \mathcal{N}_{\fkn}^{\fkm}=1\otimes \text{pr}_{1*}-\frac{\psi(\fkl)[\fkl]}{\ell^2}\otimes \text{pr}_{\ell\ell*}.
    \end{equation*}
\end{itemize}
Note that one can extend the definition of $\mathcal{N}_{\fkn}^{\fkm}$ to any pair of ideals $\fkn\mid\fkm$ by composition.

Following \cite[Sec.\,1.3.2]{CD}, if $p$ splits in $K$ and $(p,\fkf)=1$ then for any ideal $\fkn\subset\cO_K$ divisible by $\mathfrak{f}$ such that $(\fkn,\ppbar)=1$, the maximal ideal of $\mathbb{T}'(N)$ defined by the kernel of the composition
\[
\mathbb{T}'(N)\overset{\phi_\fkn}\longrightarrow\cO[H_\fkn]\overset{{\rm aug}}\longrightarrow\cO\rightarrow\cO/\mathfrak{P},
\]
is non-Eisenstein, $p$-ordinary, and $p$-distinguished. 

We finish this section by extracting a crucial result in \cite{LLZ-K} in the case where $p$ splits in $K$. This will be used later to prove the norm relation of our Euler system.

\begin{thm}[Corollary 5.2.6 in \cite{LLZ-K}]\label{norm1}
Assume that $(p,\fkf)=1$. Let $\mathcal{A}$ be the set of ideals $\fkm\subset\cO_K$  with $(\fkm,\ppbar)=1$, and put $\mathcal{A}_{\fkf}=\{\fkf\fkm\colon\fkm\in\mathcal{A}\}$. Given  $\fkn\in \mathcal{A}_{\mathfrak{f}}$, there is a $G_{\Q}$-equivariant isomorphism of $\mathcal{O}[H_{\fkn}^{(p)}]$-modules 
\begin{equation*}
    \nu_{\fkn}:\mathcal{O}[H_{\fkn}^{(p)}]\otimes_{\mathbb{T}'(N)\otimes\Z_p,\phi_{\fkn}}H^1_{\et}(Y_1(N)_{\overline{\Q}},\Z_p(1))\overset{\cong}{\longrightarrow}{\rm Ind}^{\Q}_{K(\fkn)}\mathcal{O}(\psi_{\mathfrak{P}}^{-1}).
\end{equation*}
 Furthermore, for any $\fkn,\fkm\in \mathcal{A}_\mathfrak{f}$ with $\fkn\mid\fkm$, the following diagram commutes:
\begin{equation*}
    \xymatrix{
\mathcal{O}[H_{\fkm}^{(p)}]\otimes_{\mathbb{T}'(M)\otimes\Z_p,\phi_{\fkm}}H^1_{\et}(Y_1(M)_{\overline{\Q}},\Z_p(1)) \ar[d]_{\mathcal{N}_{\fkn}^{\fkm}} \ar[r]^(.7){\nu_{\fkm}}_(.7){\cong} & {\rm Ind}^{\Q}_{K(\fkm)}\mathcal{O}(\psi_{\mathfrak{P}}^{-1}) \ar[d]_{\mathrm{Norm}_{\fkn}^{\fkm}} \\
\mathcal{O}[H_{\fkn}^{(p)}]\otimes_{\mathbb{T}'(N)\otimes\Z_p,\phi_{\fkn}}H^1_{\et}(Y_1(N)_{\overline{\Q}},\Z_p(1)) \ar[r]^(.7){ \nu_{\fkn}}_(.7){\cong} & {\rm Ind}^{\Q}_{K(\fkn)}\mathcal{O}(\psi_{\mathfrak{P}}^{-1}), }
\end{equation*}
where $\mathrm{Norm}_{\fkn}^{\fkm}$ is the natural norm map of the induced representations.
\end{thm}

\section{The construction}\label{sec:main-thms}

For a newform $f$ and two Hecke characters $\psi_1, \psi_2$ of $2$ distinct imaginary quadratic fields $K_1, K_2$ respectively, using the results from \cite{CD}, \cite{BSV} and \cite{LLZ-K} recalled in the preceding section, we construct a family of cohomology classes for $f\otimes\tilde{\psi}_1\tilde{\psi}_2$ defined over ring class field extensions of $K_0$, which is the compositum of $K_1$ and $K_2$, and prove that they satisfy the norm relations of an anticyclotomic Euler system. Following \cite[Sec.\,2]{CD}, we first give the construction and show the tame norm relations in the case where $(f,\theta_{\psi_1},\theta_{\psi_2})$ have weights $(2,2,2)$. Then by varying the diagonal cycle classes in Hida families we extend the construction to more general weights and prove the wild norm relations.

Throughout this section  we consider the following  set-up:
\begin{enumerate}
    \item Let $f\in S_k(\Gamma_0(N_f))$ be a newform of weight $k\geq 2$.
    \item Let $K_1/\Q$ be an imaginary quadratic field of discriminant $D_1$ coprime to $N_f$. Let $\psi_1$ be a Hecke character of $K_1$ of infinity type $(1-k_1,0)$, with $k_1\geq 1$, and modulus  $\fkf_1$.
    \item  Let $K_2/\Q$ be an imaginary quadratic field of discriminant $D_2\neq D_1$ and coprime to $N_f$. Let $\psi_2$ be a Hecke character of $K_2$ of infinity type $(1-k_2,0)$, with $k_2\geq 1$, and modulus  $\fkf_2$.
    \item Denote by $\epsilon_{K_i}$ the quadratic character attached to the quadratic field $K_i$ for $i\in\{1,2\}$.
    \item Let $K_0$ be the compositum of $K_1$ and $K_2$. Since $K_0$ is a biquadratic field, we can consider $K_3$, the unique real quadratic field inside $K_0$. 
    \item Let $\tilde{\psi}_i$ be the Hecke character of $K_0$, obtained by composing $\mathbb{A}_{K_0}^{\times}\xrightarrow{\mathbb{N}_{K_0/K_i}}\mathbb{A}_{K_i}^{\times}\xrightarrow{\psi_i}\mathbb{C}$ for $i\in\{1,2\}$.
    \item Denote by 
    \[\theta_{\psi_i}\in S_{k_i}(N_{\psi_i},\chi_{\psi_i}\epsilon_{K_i}) \]
    the associated theta series, where $N_{\psi_i}=N_{K_i/\Q}(\fkf_i)\cdot{\rm disc}(K_i/\Q)$ and $\chi_{\psi_i}$ is the Dirichlet character modulo $N_{K_i/\Q}(\fkf_i)$ defined by $\psi_i((n))=n^{k_i-1}\chi_{\psi_i}(n)$ for all integers $n$ prime to $N_{K_i/\Q}(\fkf_i)$ ($i\in\{1,2\}$). 
    \item We assume the self-duality condition
    \begin{equation}\label{eq:self-dual}
{\chi_{\psi_1}\epsilon_{K_1}\chi_{\psi_2}\epsilon_{K_2}=1}.
    \end{equation}
   % In particular, since $k$ is even by hypothesis, condition (\ref{eq:self-dual}) implies that $k_1\equiv k_2\,({\rm mod}\,2)$.  ???
\end{enumerate}

Let $L/K_0$ be a finite extension containing the Fourier coefficients of $f$, $\theta_{\psi_1}$, and $\theta_{\psi_2}$. Let $p\geq 5$ be a prime that splits in $K_0$ and such that $(p,N_fN_{\psi_1}N_{\psi_2})=1$, and let $\mathfrak{P}\vert\mathfrak{p}$ be the prime of  $L/K_0$ above $p$ determined by a fixed embedding $i_p:\overline{\Q}\hookrightarrow\overline{\Q}_p$. Finally, let $L_\mathfrak{P}$ be the completion of $L$ at $\mathfrak{P}$, and denote by $\cO$ the ring of integers of $L_{\mathfrak{P}}$.

\subsubsection{Digression to primes decomposition and 
 the top left-corner notations}\label{subsec:prime-decomp} Let $\ell$ be a split prime in $K_0$ i.e. $(\ell)=\cL_1\cL_2\cL_3\cL_4$. We can write $(\ell) = \la_1\bar{\la}_1$ and $(\ell) = \la_2\bar{\la}_2$ in $K_1$ and $K_2$ respectively. Note that $\ell$ also splits in $K_3$ as $\la_3\Tilde{\la}_3$, where the tilde corresponds to the nontrivial element generating the Galois group $\mathrm{Gal}(K_3/\Q)$.

 Let $\tau_i$ to be the generator of $\mathrm{Gal}(K_0/K_i)$ for $i=\{1,2,3\}$ then we have $\tau_3=\tau_1\tau_2$ (this is the complex conjugation on $K_3$). Due to the Galois group action on primes lying above $\ell$, we can further assume that:
\[\begin{split}
    \la_1=\cL_1\cL_4,\qquad \bar{\la}_1=\cL_3\cL_2,\qquad  \la_2=\cL_1\cL_3,\qquad \bar{\la}_2=\cL_2\cL_4, \\
     \la_3 = \cL_4\cL_3\qquad (\text{so } \la_3|\la_1\la_2),\qquad\text{and }\qquad \Tilde{\la}_3 = \cL_1\cL_2,
\end{split}\]
where
\begin{equation*}
    \cL_4=\tau_1\cL_1,\qquad \cL_3=\tau_2\cL_1, \qquad \cL_2=\tau_3\cL_1=\tau_1\tau_2\cL_1.
\end{equation*}

Denote by $\mathcal{L}$ the set of primes $\la_3\subset\cO_{K_3}$, where $\ell=N_{K_3/\Q}(\la_3)$ is prime to $p$ and $\ell$ splits in $K_0$. Let $\mathcal{N}$ be the set of squarefree products of primes inside $\mathcal{L}$ such that its norm down to $\Q$ is still squarefree. For such $\la_3$, we can choose $\la_1\subset\cO_{K_1}$ and $\la_2\subset\cO_{K_2}$ as above such that $\la_3|\la_1\la_2$.

Let $\mu_3\in \mathcal{N}$ and $N_{K_3/\Q}(\mu_3)=m$. Then its norm $m=\prod_i \ell_i$ will be a product of split primes $\ell_i$ in $K_0$. Similarly, we can decompose $(m)=\cM_1\cM_2\cM_3\cM_4$, $(m)=\mu_1\bar{\mu}_1$, $(m)=\mu_2\bar{\mu}_2$, $(m)=\mu_3\Tilde{\mu}_3$ as a decomposition inside $K_0$, $K_1$, $K_2$ and $K_3$ respectively, where we can have the following decomposition:
\[\begin{split}
    \mu_1=\cM_1\cM_4,\qquad \bar{\mu}_1=\cM_3\cM_2,\qquad  \mu_2=\cM_1\cM_3,\qquad \bar{\mu}_2=\cM_2\cM_4, \\
     \mu_3 = \cM_4\cM_3\qquad (\text{so } \mu_3|\mu_1\mu_2),\qquad\text{and }\qquad \Tilde{\mu}_3 = \cM_1\cM_2.
\end{split}\]
Here, for every $i$, $\cM_j=\prod_{i}\cL_{j,i}$, $\ell_i=\prod_j \cL_{j,i}$, for $1\le j\le 4$, $\mu_j=\prod_i \la_{j,i}$ for every $j\in\{1,2,3\}$.

For each $i\in \{0,1,2\}$, we denote $^{i}K_{\fkn_i}$ as the ray class field of $K_i$ with conductor $\fkn_i$ (an integral ideal inside $\mathcal{O}_{K_i})$, and let $^{i}H_{\fkn_i}$ be the ray class group of $K_i$ modulo $\fkn_i$. Let $K_i(\fkn_i)$ be the largest $p$-subextension of $K_i$ contained in $^{i}K_{\fkn_i}$, so ${\rm Gal}(K_i(\fkn_i)/K_i)\cong {^{i}H_{\fkn_i}^{(p)}}$ is the largest $p$-power quotient of $^{i}H_{\fkn_i}$.

\subsection{Construction in weight $(2,2,2)$ and tame norm relation}
\label{subsec:tame}
 Suppose in this subsection that $(k,k_1,k_2)=(2,2,2)$. Let $N=\text{lcm}(N_f,N_{\psi_1},N_{\psi_2})$. Following Section $2.1$ of \cite{CD}, which is based on the diagonal classes in the triple product of modular curves \cite[Sec.\,3]{BSV}, we have cohomology classes:
\begin{equation}\label{kappa3}
\mathcal{Z}_m^{(1)}:=\tilde{\kappa}_{m}^{(3)}\in H^1\bigl(\Q,H^1_{\et}(Y_1(N)_{\overline{\Q}},\Z_p(1))\otimes H^1_{\et}(Y_1(N_{\psi_1}m)_{\overline{\Q}},\Z_p(1))\otimes H^1_{\et}(Y_1(N_{\psi_2}m)_{\overline{\Q}},\Z_p(1))(-1)\bigr).
\end{equation}
for every positive integer $m$. One then chooses a test vector $\breve{f}\in S_2(N)[f]$. As noted in op. cit., the maps used to construct $\mathcal{Z}_m^{(1)}$ are compatible with correspondences. This allows one to tensor them with $\cO$ and obtain:
\begin{align*}
\mathcal{Z}_{\mu_3}^{(1)}\in H^1\bigl(\Q,T_{f}^{\vee}&\otimes H^1_{\et}(Y_1(N_{\psi_1}m)_{\overline{\Q}},\Z_p(1))\otimes_{\mathbb{T}'(N_{\psi_1}m)}\mathcal{O}[^{1}H_{\mathfrak{f}_1\mu_1}^{(p)}]\\
 &\quad\otimes H^1_{\et}(Y_1(N_{\psi_2}m)_{\overline{\Q}},\Z_p(1))\otimes_{\mathbb{T}'(N_{\psi_2}m)}\mathcal{O}[^{2}H_{\mathfrak{f}_2\mu_2}^{(p)}](-1)\bigr).
\end{align*}
Here, the chosen $\breve{f}$ is used to take the image under the  projection $H^1_{\et}(Y_{\overline{\Q}},\Z_p(1))\rightarrow T_f^\vee$ in the first factor. The tensor products are taken from Proposition\,\ref{phi_n} 
\[
\phi_{\mathfrak{f}_1\mu_1}:\mathbb{T}'(N_{\psi_1}m)\rightarrow\mathcal{O}[^{1}H_{\mathfrak{f}_1\mu_1}^{(p)}],\quad 
\phi_{\mathfrak{f}_2\mu_2}:\mathbb{T}'(N_{\psi_2}m)\rightarrow\mathcal{O}[^{2}H_{\mathfrak{f}_2\mu_2}^{(p)}]
\]
with respect to two distinct imaginary quadratic fields $K_1$ and $K_2$, respectively .  

Via the isomorphisms from Proposition \ref{norm1} with respect to $2$ distinct imaginary quadratic fields:
\begin{align*}
\nu_{\fkf_1\mu_1}:H^1_{\et}(Y_1(N_{\psi_1}m)_{\overline{\Q}},\Z_p(1))\otimes_{\mathbb{T}'(N_{\psi_1}m)}\mathcal{O}[^{1}H_{\mathfrak{f}_1\mu_1}^{(p)}]&\xrightarrow{\sim}{\rm Ind}^{\Q}_{K_1(\fkf_1\mu_1)}\mathcal{O}(\psi_{1}^{-1}),\\
\nu_{\fkf_2\mu_2}:H^1_{\et}(Y_1(N_{\psi_2}m)_{\overline{\Q}},\Z_p(1))\otimes_{\mathbb{T}'(N_{\psi_2}m)}\mathcal{O}[^{2}H_{\mathfrak{f}_2\mu_2}^{(p)}]&\xrightarrow{\sim}{\rm Ind}^{\Q}_{K_2(\fkf_2\mu_2)}\mathcal{O}(\psi_{2}^{-1}),
\end{align*}
one then obtains a cohomology class in
\[
H^1\bigl(\Q,T_f^{\vee}\otimes_{\mathcal{O}}\text{Ind}^{\Q}_{K_1(\fkf_1\mu_1)}\mathcal{O}(\psi_1^{-1})\otimes_{\mathcal{O}}\text{Ind}^{\Q}_{K_2(\fkf_2\mu_2)}\mathcal{O}(\psi_2^{-1})(-1)\bigr),
\]
which under the maps $
^{1}H_{\fkf_1\mu_1}\twoheadrightarrow {^{1}H_{\mu_1}}$ and $^{2}H_{\fkf_2\mu_2}\twoheadrightarrow {^{2}H_{\mu_2}}$ can be projected to a class
\begin{equation}\label{eq:inf-res-decomp}
\mathcal{Z}_{\mu_3}^{(2)}\in H^1\bigl(\Q,T_f^{\vee}\otimes_{\mathcal{O}}{\rm Ind}^{\Q}_{K_1}\mathcal{O}_{\psi_1^{-1}}[^{1}H_{\mu_1}^{(p)}]\otimes_{\mathcal{O}}{\rm Ind}^{\Q}_{K_2}\mathcal{O}_{\psi_2^{-1}}[^{2}H_{\mu_2}^{(p)}](-1)\bigr).
\end{equation}
Note that the group cohomology can be rewritten as
\[
H^1\bigl(\Q,T_f^{\vee}\otimes_{\mathcal{O}}{\rm Ind}^{\Q}_{K_0}\mathcal{O}_{\tilde{\psi}_1^{-1}\tilde{\psi}_2^{-1}}[{^{1}H_{\mu_1}^{(p)}}\times {^{2}H_{\mu_2}^{(p)}}](-1)\bigr),\]
which by Shapiro's lemma gives us elements:
\begin{equation}
    \label{eqn:pre-shapiro}\mathcal{Z}_{\mu_3}^{(3)}\in
H^1\bigl(K_0,T_f^{\vee}\otimes_{\mathcal{O}}\mathcal{O}_{\tilde{\psi}_1^{-1}\tilde{\psi}_2^{-1}}[{^{1}H_{\mu_1}^{(p)}}\times {^{2}H_{\mu_2}^{(p)}}](-1)\bigr).
\end{equation}

\subsubsection{Projection to ring class groups}

Recall the fundamental exact sequence for ray class groups:
\[
\xymatrix{
\cO_{K_i}^\times\ar[r]&(\cO_{K_i}/\mu_i\cO_{K_i})^\times\ar[r]&^{i}H_{\mu_i}\ar[r]& ^{i}H_1\ar[r]&1,
}
\]
where $i\in\{1,2\}$. Assume that $p\nmid 6h_{K_0}$, where $h_{K_0}$ is the class number of $K_0$. Note that by \cite[Thm.\,74]{Frohlich_Taylor_1991}, we have
\[p\nmid h_{K_i},\]
which is the class number of $K_i$ for $i\in\{1,2,3\}$. Taking the $p$-primary parts of the above exact sequence induces two isomorphisms
\begin{align*}
    &^{1}H_{\mu_1}^{(p)}\xrightarrow{\simeq} (\cO_{K_1}/\mu_1\cO_{K_1})^{\times,(p)} \xrightarrow{\simeq} (\cO_{K_0}/\cM_4\cO_{K_0})^{\times,(p)} \\
   & ^{2}H_{\mu_2}^{(p)}\xrightarrow{\simeq} (\cO_{K_2}/\mu_2\cO_{K_2})^{\times,(p)} \xrightarrow{\simeq} (\cO_{K_0}/\cM_3\cO_{K_0})^{\times,(p)}
\end{align*}
and hence the following projection:
\begin{equation}\label{eq:H-m-decompose}
    ^{1}H_{\mu_1}^{(p)}\times {^{2}H}_{\mu_2}^{(p)} \xrightarrow{\simeq} (\cO_{K_0}/\mu_3\cO_{K_0})^{\times,(p)} \twoheadrightarrow {^{0}H_{\mu_3}^{(p)}}.
\end{equation}

Recall that $K_3$ is the totally real field sitting inside a CM field $K_0$. Given an integral ideal $\fkn$ of $K_3$, let $H[\fkn]$ be the ring class group of $K_0$ of conductor $\fkn$, so 
$H[\fkn]\simeq{\rm Pic}(\cO_\fkn)$ under the Artin reciprocity map, where $\cO_\fkn=\cO_{K_3}+\fkn\cO_{K_0}$ is the order of $K_0$ of conductor $\fkn$. Let $H[\fkn]^{(p)}$ be the maximal $p$-power quotient of $H[\fkn]$, and denote by $K_0[\fkn]$ be the maximal $p$-extension inside the ring class field of $K_0$ of conductor $\fkn$, i.e. $H[\fkn]^{(p)}={\rm Gal}(K_0[\fkn]/K_0)$. Note that for the ring class groups and fields of $K_0$, we drop the upper left corner $0$ notation.

\begin{prop}\label{ringclass} Suppose $p\nmid 6h_{K_0}$ and $\mu_3$ is a squarefree ideal of $\cO_{K_3}$ of norm $m$, where $m$ is divisible only by primes that are split in $K_0$. Using (\ref{eq:H-m-decompose}), we have the following short exact sequence:
\begin{comment}
    \[
\xymatrix{
(\cO_{K_0}/\mu_3\cO_{K_0})^{\times,(p)} \ar[r]^{\simeq}&{^{1}H_{\mu_1}^{(p)}}\times {^{2}H_{\mu_2}^{(p)}}\ar@{->>}[d]^{w}& & \\
(\cO_{K_3}/\mu_3\cO_{K_3})^{\times, (p)}\ar[r]\ar[u]^{\Delta}&{^{0}H_{\mu_3}^{(p)}}\ar[r]^{e}&H[\mu_3]^{(p)}\ar[r]&1.
}
\]
\end{comment}
\[
1\rightarrow(\cO_{K_3}/\mu_3\cO_{K_3})^{\times, (p)}\xrightarrow{\Delta}{^{1}H}_{\mu_1}^{(p)}\times {^{2}H}_{\mu_2}^{(p)}\xrightarrow{e\circ w} H[\mu_3]^{(p)}\rightarrow 1.
\]
Here, the map $\Delta$ uses the identifications 
\[(\cO_{K_3}/\mu_3\cO_{K_3})^{\times,(p)} \simeq (\cO_{K_0}/\cM_4\cO_{K_0})^{\times,(p)},\qquad (\cO_{K_3}/\mu_3\cO_{K_3})^{\times,(p)} \simeq (\cO_{K_0}/\cM_3\cO_{K_0})^{\times,(p)}.
\]
Moreover, if $(\ell)=\cL_1\cL_2\cL_3\cL_4$ is a prime that splits in $K_0$ and is coprime to $m$, the projection $e\circ w$ (defined in the proof below) sends 
\begin{align*}
      \quad  [\la_1]\times[\bar{\la}_2]&\mapsto \mathrm{Frob}_{\cL_4/\la_3} 
    \end{align*}
where $\mathrm{Frob}_{\cL_4/\la_3} $ is the Frobenius element of $\cL_4$ in $H[\mu_3]^{(p)}$. 
\end{prop}

\begin{proof}
We have the following exact diagram: 
\[
\xymatrix{
\cO_{K_0}^\times\ar[r]\ar[d]&(\cO_{K_0}/\mu_3\cO_{K_0})^\times\ar[r]\ar[d]&^{0}H_{\mu_3}\ar[r]\ar[d]&^{0}H_1\ar[r]\ar[d]&1\\
\cO_{K_0}^\times/\cO_{K_3}^\times\ar[r]&(\cO_{K_0}/\mu_3\cO_{K_0})^\times/(\cO_{K_3}/\mu_3\cO_{K_3})^\times\ar[r]&H[\mu_3]\ar[r]&^{0}H_1\ar[r]&1.
}
\]
Taking the $p-$part of this, using the assumption that $p\nmid 6h_{K_0}$ together with the fact that $|\cO_{K_0}^\times/\cO_{K_3}^\times|$ is a power of $2$, we obtain the following diagram:
\[
\xymatrix{
\cO_{K_0}^{\times}\otimes \frac{\Q_p}{\Z_p}\ar[r]\ar[d]&(\cO_{K_0}/\mu_3\cO_{K_0})^{\times,(p)}\ar[r]^{w}\ar[d]&^{0}H_{\mu_3}^{(p)}\ar[r]\ar[d]^{e}&1\\
1\ar[r]&(\cO_{K_0}/\mu_3\cO_{K_0})^{\times,(p)}/(\cO_{K_3}/\mu_3\cO_{K_3})^{\times,(p)}\ar[r]&H[\mu_3]^{(p)}\ar[r]&1.
}
\]
Using the middle arrow and the identification \ref{eq:H-m-decompose}, we can show the first exact sequence.

One can show the second part by noting that $\cL_4$ is a prime of $K_0$ lying above both $\la_1$ (a prime of $K_1$) and $\bar{\la}_2$ (a prime of $K_2$).
\end{proof}

In the same setting as Proposition~\ref{ringclass}, we can consider the image of $\mathcal{Z}_{\mu_3}^{(3)}$ from \ref{eqn:pre-shapiro} under the composition $e\circ w$. This results in the class 
\[
\mathcal{Z}_{\mu_3}^{(4)}\in  H^1\bigl(K_0,T_f^{\vee}\otimes_{\mathcal{O}}\mathcal{O}_{\tilde{\psi}_1^{-1}\tilde{\psi}_2^{-1}}[H[\mu_3]^{(p)}](-1)\bigr).
\]
By Shapiro's lemma, its image under the isomorphism 
\[H^1\bigl(K_0,T_f^{\vee}\otimes_{\mathcal{O}}\mathcal{O}_{\tilde{\psi}_1^{-1}\tilde{\psi}_2^{-1}}[H[\mu_3]^{(p)}](-1)\bigr)\simeq H^1\bigl(K_0[\mu_3],T_f^{\vee}(\tilde{\psi}_{1}^{-1}\tilde{\psi}_{2}^{-1})(-1)\bigr)
\]
then defines
\[
\mathcal{Z}_{\mu_3}^{(5)}\in H^1\bigl(K_0[\mu_3],T_f^{\vee}(\psi_{1}^{-1}\psi_{2}^{-1})(-1)\bigr).
\]

The next lemma is in the same vein with \cite[Lem.\,2.1.3]{CD}:
\begin{lem}\label{keydiagram}
Let $\la_3$ be a split prime in $K_3$ of norm $\ell$ coprime to $mp$, where $\ell$ splits in $K_0$. The following diagram commutes:
\[
\xymatrix{
 {\rm Ind}^{\Q}_{K_1}\mathcal{O}_{\psi_{1}^{-1}}[^{1}H_{\mu_1\la_1}^{(p)}]\otimes_{\mathcal{O}}{\rm Ind}^{\Q}_{K_2}\mathcal{O}_{\psi_{2}^{-1}}[^{2}H_{\mu_2\la_2}^{(p)}] \ar[d]^{\mathrm{Norm}^{\mu_1\la_1}_{\mu_1}\otimes\mathrm{Norm}^{\mu_2\la_2}_{\mu_2}} \ar[r] &{\rm Ind}^{K_3}_{K_0}\mathcal{O}_{\tilde{\psi}_{1}^{-1}\tilde{\psi}_{2}^{-1}}[H[\mu_3\la_3]^{(p)}] \ar[d]^{\mathrm{Norm}^{\mu_3\la_3}_{\mu_3}} \\
{\rm Ind}^{\Q}_{K_1}\mathcal{O}_{\psi_{1}^{-1}}[^{1}H_{\mu_1}^{(p)}]\otimes_{\mathcal{O}}{\rm Ind}^{\Q}_{K_2}\mathcal{O}_{\psi_{2}^{-1}}[^{2}H_{\mu_2}^{(p)}] \ar[r] & {\rm Ind}^{K_3}_{K_0}\mathcal{O}_{\tilde{\psi}_{1}^{-1}\tilde{\psi}_{2}^{-1}}[H[\mu_3]^{(p)}],}
\]
where the norm maps are the natural ones, and where the horizonal arrows are given by the composition $e\circ w$ in  (\ref{ringclass}).
\end{lem}

Using this Lemma \ref{keydiagram}, one can show the following Proposition~\ref{wrongnorm}. Similar to \cite[Prop ${2.2.1}$]{CD}, this is the key result for the construction of our anticyclotomic Euler system for $T_f^\vee(\tilde{\psi}_1^{-1}\tilde{\psi}_2^{-1})(-1)$ over the biquadratic field $K_0$.

\begin{prop}\label{wrongnorm}
The family $\{\mathcal{Z}_{\mu_3}^{(5)}\}$ satisfies the following norm relation:
\begin{align*}
{\rm Norm}_{K_0[\mu_3]}^{K_0[\mu_3\la_3]}(\mathcal{Z}_{\mu_3\la_3}^{(5)})=(\ell-1)\bigg(&a_\ell(f)-\frac{\psi_1(\la_1)\psi_2(\bar{\la}_2)}{\ell}([\la_1]\times [\bar{\la}_2])-\frac{\psi_1(\bar{\la}_1)\psi_2(\la_2)}{\ell}([\bar{\la}_1]\times [\la_2])\\
   &+(1-\ell)\frac{\psi_1(\la_1)\psi_2(\la_2)}{\ell^2}([\la_1]\times[\la_2])\bigg)
    (\mathcal{Z}_{\mu_3}^{(5)}).
\end{align*}
\end{prop}

\begin{proof}
As in the proof of \cite[Prop ${2.2.1}$]{CD}, one has
\begin{align*}
&(1\otimes\mathcal{N}_{\mu_1}^{\mu_1\la_1}\otimes\mathcal{N}_{\mu_2}^{\mu_2\la_2})(\mathcal{Z}^{(5)}_{\mu_3\la_3})\\
&=(\ell-1)\bigg((T_\ell,1,1)-\frac{\psi_1(\la_1)[\la_1]}{\ell}(1,1,T_\ell')-\frac{\psi_2(\la_2)[\la_2]}{\ell}(1,T_\ell',1)+\frac{\psi_1(\la_1)\psi_2(\la_2)}{\ell^2}([\la_1]\times[\la_2])(\ell+1)\bigg)(\mathcal{Z}^{(5)}_{\mu_3}) \\
&=(\ell-1)\bigg(a_\ell(f)-\frac{\psi_1(\la_1)[\la_1]}{\ell}(\psi_2(\la_2)[\la_2]+\psi_2(\bar{\la}_2)[\bar{\la}_2])-(\psi_1(\la_1)[\la_1]+\psi_1(\bar{\la}_1)[\bar{\la}_1])\frac{\psi_2(\la_2)[\la_2]}{\ell}\\
&\quad+\frac{\psi_1(\la_1)\psi_2(\la_2)}{\ell^2}([\la_1]\times[\la_2])(\ell+1)\bigg)(\mathcal{Z}^{(5)}_{\mu_3})\\
&=(\ell-1)\bigg(a_\ell(f)-\frac{\psi_1(\la_1)\psi_2(\bar{\la}_2)}{\ell}([\la_1]\times [\bar{\la}_2])-\frac{\psi_1(\bar{\la}_1)\psi_2(\la_2)}{\ell}([\bar{\la}_1]\times [\la_2])\\
&\quad+(1-\ell)\frac{\psi_1(\la_1)\psi_2(\la_2)}{\ell^2}([\la_1]\times[\la_2])\bigg)(\mathcal{Z}^{(5)}_{\mu_3}).
\end{align*} 
This implies the result via combining Theorem \ref{norm1} and Lemma \ref{keydiagram}.
\end{proof}

Following \cite[$\S{2.2}$]{CD} verbatim, which borrows ideas from \cite[$\S{1.4}$]{DR2}, we can strip out the $(\ell-1)$ factor by quotienting out the diamond operators action and obtain modified classes
\[
\mathcal{Z}_{\mu_3}^{(6)}\in H^1\bigl(K_0[\mu_3],T_f^{\vee}(\tilde{\psi}_{1}^{-1}\tilde{\psi}_{2}^{-1})(-1)\bigr).
\]
Then the term in the right-hand side of Proposition~\ref{wrongnorm} can be massaged to agree with the local Euler factor at $\cL_2$ of the Galois representation $[T_f^\vee(\tilde{\psi}_1^{-1}\tilde{\psi}_2^{-1})(-1)]^\vee(1)=T_f(\tilde{\psi}_1\tilde{\psi}_2)(2)$, giving the correct norm relations:
\begin{comment}
    
\begin{prop}\label{correctnorm}
Suppose $f$ is non-Eisenstein modulo $\mathfrak{P}$.  
Let $\mu_3$ be an ideal of $\cO_{K_3}$ of norm $m$, where $m$ is divisible only by primes split in $K_0$. Let $\la_3$ be a prime of $\cO_{K_3}$ of norm $\ell$ where $\ell$ splits in $K_0$ and $\ell$ is coprime to $mp$. Then
\begin{align*}
{\rm Norm}_{K_0[\mu_3]}^{K_0[\mu_3\la_3]}({\kappa}_{f,\psi_1,\psi_2,\mu_3\la_3})=\bigg(&a_\ell(f)-\frac{\psi_1(\la_1)\psi_2(\bar{\la}_2)}{\ell}([\la_1]\times [\bar{\la}_2])-\frac{\psi_1(\bar{\la}_1)\psi_2(\la_2)}{\ell}([\bar{\la}_1]\times [\la_2])\\
   &+(1-\ell)\frac{\psi_1(\la_1)\psi_2(\la_2)}{\ell^2}([\la_1]\times[\la_2])\bigg)
    ({\kappa}_{f,\psi_1,\psi_2,\mu_3}).
\end{align*}
\end{prop}

Thus we arrive at the following theorem:
\end{comment}

\begin{thm}\label{maintheorem1}
Suppose $p\nmid 6h_{K_0}$ and $f$ is non-Eisenstein modulo $\mathfrak{P}$. Let $\mu_3\in\mathcal{N}$ be a squarefree ideal of $\cO_{K_3}$. Then there exists a collection of cohomology classes
\[
\mathcal{Z}_{\mu_3}\in H^1\bigl(K_0[\mu_3],T_f^{\vee}(\tilde{\psi}_{1}^{-1}\tilde{\psi}_{2}^{-1})(-1)\bigr)
\]
such that for every split prime $\la_3\in\mathcal{L}$ of $\cO_{K_3}$ of norm $\ell$ with $(\ell,m)=1$, we have the norm relation
\begin{equation*}
    \mathrm{Norm}_{K_0[\mu_3]}^{K_0[\mu_3\la_3]}(\mathcal{Z}_{\mu_3\la_3})=P_{\cL_4}(\mathrm{Frob}_{\cL_4/\la_3})(\mathcal{Z}_{\mu_3}),
\end{equation*}
where $P_{\cL_4}(X)=\det(1-X\cdot \mathrm{Frob}_{\cL_4/\la_3}\,|\,T_f(\tilde{\psi}_1\tilde{\psi}_2)(2))$.
\end{thm}
\begin{comment}
    \begin{proof}
Denote by $Q_{\cL_2}$ the factor appearing in the right-hand side of Proposition \ref{correctnorm}. Recalling that $[\la_1]\times[\bar{\la}_2]$ corresponds to ${\rm Frob}_{\cL_2/\la_3}\in H[\mu_3]^{(p)}$ under the map $e\circ w$ of Proposition \ref{ringclass}, we find the following congruences as endomorphisms of $H^1(K_0[\mu_3],T_f^{\vee}(\tilde{\psi}_{1}^{-1}\tilde{\psi}_{2}^{-1})(-1))$:
\begin{align*}
&\quad -\psi_1(\la_1)\psi_2(\bar{\la}_2)([\la_1]\times [\bar{\la}_2])\cdot Q_{\cL_2}\\
&\equiv
-a_\ell(f)\psi_1(\la_1)\psi_2(\bar{\la}_2){\rm Frob}_{\cL_2/\la_3}+\frac{\psi_1(\la_1)^2\psi_2(\bar{\la}_2)^2}{\ell}{\rm Frob}_{\cL_2/\la_3}^2+\frac{\psi_1\psi_2((\ell))}{\ell}([\ell]\times[\ell])\\
&\equiv P_{\cL_2}(\mathrm{Frob}_{\cL_2/\la_3})\pmod{\ell-1},
\end{align*}
using the relation $\psi_1\psi_2((\ell))=\chi_{\psi_1}\chi_{\psi_2}(\ell)\ell^2=\epsilon_{K_1}(\ell)^{-1}\epsilon_{K_2}(\ell)^{-1}\ell^2=\ell^2$ since $\ell$ splits in $K_0$, $\psi_1(\la_1)\psi_2(\bar{\la}_2)=\tilde{\psi}_1(\cL_2)\tilde{\psi}_2(\cL_2)$, and the fact that $[\ell]\times[\ell]$ maps to the identity element inside the ring class group. Therefore, by Lemma 9.6.1 and 9.6.3 in \cite{Rubin-ES}, the existence of classes $z_{f,\psi_1,\psi_2,\mu_3}$ satisfying the stated norm relations, which follows from Proposition~\ref{correctnorm}.
\end{proof}

\end{comment}

\begin{proof} The proof is parallel to the one of \cite[Thm.\,2.2.7]{CD}.
First notes that $[\la_1]\times[\bar{\la}_2]$ corresponds to ${\rm Frob}_{\cL_4/\la_3}\in H[\mu_3]^{(p)}$ under the map $e\circ w$ of Proposition \ref{ringclass}. One then multiplies the class $\mathcal{Z}_{\mu_3\la_3}^{(6)}$ with
 $-\psi_1(\la_1)\psi_2(\bar{\la}_2)([\la_1]\times [\bar{\la}_2])$. From $\psi_1\psi_2((\ell))=\chi_{\psi_1}\chi_{\psi_2}(\ell)\ell^2=\epsilon_{K_1}(\ell)^{-1}\epsilon_{K_2}(\ell)^{-1}\ell^2=\ell^2$ since $\ell$ splits in $K_0$, $\psi_1(\la_1)\psi_2(\bar{\la}_2)=\tilde{\psi}_1(\cL_4)\tilde{\psi}_2(\cL_4)$, and the fact that $[\ell]\times[\ell]$ maps to the identity element inside the ring class group together with Lemma 9.6.1 and 9.6.3 in \cite{Rubin-ES}, the result follows from the explicit formula of Proposition~\ref{wrongnorm}.
\end{proof}

\subsection{Construction for general weights and wild norm relations}\label{wildnormsetup}

We now extend the above construction to other weights $(k,k_1,k_2)\in\mathbb{Z}_{ \ge 1}^3$ following \cite[Sec.\,2.3]{CD}. Then we show that the constructed cohomology classes also satisfy the wild norm relations for the anticyclotomic $\Z_p^2$-extension of $K_0$. 

First, we assume that $p\nmid 6h_{K_0}$. Assume further that $p \text{ splits in } K_0$ i.e.
\begin{equation*}
    (p)=\cP_1\cP_2\cP_3\cP_4 \text{ in } K_0,
\end{equation*}
and $\cP_2=\tau_3\cP_1,\;\cP_3=\tau_2\cP_1,\;\cP_4=\tau_1\cP_1,$. Hence
\begin{equation*}
\textrm{$(p)=\pp_1\ppbar_1$  in $K_1$,} \textrm{ $(p)=\pp_2\ppbar_2$ in $K_2$,} 
\end{equation*}
with $\cP_1$ the prime of $K_0$ above $p$ induced by our fixed embedding $i_p:\overline{\Q}\hookrightarrow\overline{\Q}_p$, and $\cP_1$ lies above $\pp_i$, the prime of $K_i$ for $i\in\{1,2\}$. Note that the numbering here is parallel to our convention in Section \ref{subsec:prime-decomp}.

Let $\Gamma_{\pp_i}^{K_i}$ be the Galois group of the unique $\Z_p$-extension of $K_i$ unramified outside $\pp_i$. There exists a unique Hecke character $\psi_{0,i}$ of $K_i$ of infinity type $(-1,0)$ and conductor $\pp_i$ such that its $p$-adic avatar factors through $\Gamma_{\pp_i}^{K_i}$. The character $\psi_i$ fixed at the beginning of this section can be decomposed as 
\[
\psi_i=\ch_i {\psi_{0,i}^{k_i-1}},
\]
where $\ch_i$ is a ray class character of $K$ of conductor dividing $\fkf_i\pp_i$. Noting that $\Gamma_{\pp_i}^{K_i}$ is a quotient of $^iH_{\fkf_i\pp_i^\infty}$ allows us to view $\psi_{0,i}$ and $\ch_i$ as characters of $^iH_{\fkf_i\pp_i^\infty}$. The formal $q$-expansion
\[
\boldsymbol{\theta}_{\ch_i}(q)=\sum_{(\mathfrak{a},\fkf_i\pp_i)=1}\ch_i\psi_{0,i}(\mathfrak{a})[\mathfrak{a}]q^{N_{K/\Q}(\mathfrak{a})}\in\Lambda_{\pp_i}[\![q]\!],
\]
where $\Lambda_{\pp_i}=\cO[\![\Gamma_{\pp_i}^{K_i}]\!]\simeq\cO\dBr{\Gamma}$ and $\Gamma=1+p\Z_p$, is the Hida family passing through $\theta_{\psi_i}$ (the specialisation of $\boldsymbol{\theta}_{\ch_i}$ at weight $k_i$ and trivial character recovers the ordinary $p$-stabilization of $\theta_{\psi_i}$). Here we identify $\Gamma_{\pp_i}^{K_i}$ with $1+p\Z_p$ via the (geometrically normalized) local Artin map.

Let $\bff$ be the Hida family associated to $f$. Let $\bfg=\boldsymbol{\theta}_{\ch_1}$, $\bfh=\boldsymbol{\theta}_{\xi_2}$
be the CM Hida families associated to $\psi_1$ and $\psi_2$, respectively. Denote by $\kappa_f$,  $\kappa_g$, and $\kappa_h$ the Dirichlet characters modulo $p$ giving the $p$-part of the tame characters of $\bff$, $\bfg$, and $\bfh$, respectively.

Under the assumption that
${\ch_i\psi_{0,i}\not\equiv\omega\;\;({\rm mod}\,{\mathfrak{P}})}$
for $i\in\{1,2\}$, following equation $(2.17)$ from \cite[$\S 2.3$]{CD} and its notation, we have the $G_\Q$-equivariant maps 
\begin{equation}\label{eq:proj-gh}
\begin{aligned}
H^1(\Gamma(m,p),\mathcal{D}'_{\kappa_1})\otimes\cO[\![^1H_{\fkf_1\mu_1\pp_1^\infty}^{(p)}]\!]&\rightarrow{\rm Ind}_{K_1}^\Q\cO_{(\ch_1\psi_{0,1})^{-1}}[^1H_{\mu_1}^{(p)}][\![\Gamma_{\pp_1}^{K_1}]\!],\\
H^1(\Gamma(m,p),\mathcal{D}'_{\kappa_2})\otimes\cO[\![^2H_{\fkf_2\mu_2\pp_2^\infty}^{(p)}]\!]&\rightarrow{\rm Ind}_{K_2}^\Q\cO_{(\ch_2\psi_{0,2})^{-1}}[^2H_{\mu_2}^{(p)}][\![\Gamma_{\pp_2}^{K_2}]\!],
\end{aligned}
\end{equation}
where $\Gamma(m,p)=\Gamma_1(Nm)\cap\Gamma_0(p)$ is a congruence subgroup.
Focusing on the class $\boldsymbol{\kappa}_m^{(2)}$ in equation $(2.15)$ of op. cit., we first tensor it with $\cO[^1H_{\fkf_1\mu_1\pp_1^r}^{(p)}]$ and $\cO[^2H_{\fkf_2\mu_2\pp_2^r}^{(p)}]$, let $r\rightarrow\infty$, and then arrive at
\begin{align*}\label{eq:good-big-3}
\boldsymbol{\mathcal{Z}}^{(1)}_{\mu_3}\in H^1\bigl(\Q,H^1(\Gamma(1,p),\mathcal{D}'_{\kappa_f})&\hat\otimes_{\cO} (H^1(\Gamma(m,p),\mathcal{D}'_{\kappa_g})\otimes\cO[\![^1H_{\fkf_1\mu_1\pp_1^\infty}^{(p)}]\!])\\
&\quad\hat\otimes_{\cO[D_m]} (H^1(\Gamma(m,p),\mathcal{D}'_{\kappa_h})\otimes\cO[\![^2H_{\fkf_2\mu_2\pp_2^\infty}^{(p)}]\!])(2-\kappa_{f\bfg\bfh}^*)\bigr).
\end{align*}

Now choose a level-$N$ test vector for $f$, denoted as $\breve{f}$. It comes with a specialization map
\begin{equation}\label{eq:proj-f}
\pi_f:H^1(\Gamma(1,p),\mathcal{D}_{\kappa_f}')(1)\rightarrow T_f^\vee.
\end{equation}
Under the natural maps induced by (\ref{eq:proj-gh}) and (\ref{eq:proj-f}), the image of $\boldsymbol{\mathcal{Z}}^{(1)}_{\mu_3}$  is then
\[
\boldsymbol{\mathcal{Z}}^{(2)}_{\mu_3}\in H^1\bigl(\Q,T_f^\vee\otimes_{\cO}({\rm Ind}_{K_1}^\Q\cO_{(\ch_1\psi_{0,1})^{-1}}[^1H_{\mu_1}^{(p)}][\![\Gamma_{\pp_1}^{K_1}]\!])\hat\otimes_{\cO[D_m]}({\rm Ind}_{K_2}^\Q\cO_{(\ch_2\psi_{0,2})^{-1}}[^2H_{\mu_2}^{(p)}][\![\Gamma_{\pp_2}^{K_2}]\!])(-1-\kappa_{f\bfg\bfh}^*)\bigr).
\]

We first follow (\ref{eq:inf-res-decomp}) and then apply the diagonal map $e\circ w$ in Proposition \ref{ringclass}. This induces the  following class
\begin{equation}\label{eq:big-5}
\boldsymbol{\mathcal{Z}}^{(3)}_{\mu_3}\in 
H^1\bigl(K_3,T_f^\vee(1-k/2)\otimes_\cO{\rm Ind}_{K_0[\mu_3]}^{K_3}\Lambda_\cO(\tilde{\psi}_1^{-1}\tilde{\psi}_2^{-1}\tilde{\kappa}_{\rm ac,1}^{(k_1-2)/2}\tilde{\kappa}_{\rm ac,2}^{(k_2-2)/2}\boldsymbol{\kappa}_{\rm ac}^{-1})(1-(k_1+k_2)/2)\bigr).
\end{equation}
Here, for $i\in\{1,2\}$, we identify $\Ga_i^{-}=\text{Gal}(K_{i,\infty}^{-}/K_i)$ with the anti-diagonal in $(1+p\Z_p)\times (1+p\Z_p)\simeq \mathcal{O}_{K_i,\mathfrak{p}_i}^{(1)}\times \mathcal{O}_{K_i,\ppbar_i}^{(1)}$  via the geometric normalized Artin map, and define
\[\kappa_{ac,i}:\Ga_i^- \rightarrow \Z_p^{\times}, \quad \quad ((1+p)^{-1/2},(1+p)^{1/2})\mapsto (1+p)\]
(compare this with equation $(2.19)$ of \cite{CD}). 
Note that here we identify the anticyclotomic $\Z_p$ extension of $K_i$ with the unique $\Z_p$ extension of $K_i$ unramified outside $\mathfrak{p}_i$ (projection to the anticyltomic part introduces a square root ($\gamma\mapsto \gamma^{(1-c)/2}$), see also \cite[p.636]{Hida-Big-Galois}). 

We then identify the anticyclotomic $\Z_p^2$ extension $\Ga^{-}=\text{Gal}(K_{0,\infty}^{-}/K_0)$ of $K_0$ with $\Gamma_{\pp_1}^{K_1}\times \Gamma_{\pp_2}^{K_2}$ via the following diagram:
\begin{equation}
    \label{eq:G_K0^-identification}
\xymatrix{
(\mathcal{O}_{K_0,\cP_1}^{(1)}\times \mathcal{O}_{K_0,\cP_2}^{(1)})\times (\mathcal{O}_{K_0,\cP_3}^{(1)}\times \mathcal{O}_{K_0,\cP_4}^{(1)})\ar[d]\ar@{->>}[rr]\ar[d]&&\frac{(1+p\Z_p)\times(1+p\Z_p)}{\mathrm{diag}}\times \frac{(1+p\Z_p)\times(1+p\Z_p)}{\mathrm{diag}}\ar[d] \ar[r]^-{\simeq}&\Z_p^2\ar[d]\\
(\mathcal{O}_{K_1,\mathfrak{p}_1}^{(1)}\times \mathcal{O}_{K_1,\ppbar_1}^{(1)})\times (\mathcal{O}_{K_2,\mathfrak{p}_2}^{(1)}\times \mathcal{O}_{K_2,\ppbar_2}^{(1)})\ar@{->>}[rr]&&\frac{(1+p\Z_p)\times(1+p\Z_p)}{\mathrm{diag}}\times \frac{(1+p\Z_p)\times(1+p\Z_p)}{\mathrm{diag}}\ar[r]^-{\simeq}&\Z_p\times \Z_p
}
\end{equation}
Let $\La^-=\Z_p\dBr{\Ga^-}$ and define further $\bm{\kappa}_{ac}:\Ga^- \rightarrow \La^{\times}$ where
\begin{align*}
    ((1+p)^{-1/2},(1+p)^{1/2},(1+p)^{-1/2},(1+p)^{1/2})\mapsto [(1+p),(1+p)].
\end{align*}
Given an $\cO$-lattice $T$  inside a $G_{K_0}$-representation $V$, Shapiro's lemma allows us to write
\[
H^1\bigl(K_0,T\hat\otimes_{\cO}\Lambda_\cO^-(\boldsymbol{\kappa}_{\rm ac}^{-1})\bigr)\simeq H^1_{\rm Iw}(K_0[p^\infty],T),
\]
where $H^1_{\rm Iw}(K_0[p^\infty],T):=\varprojlim_{r,s} H^1(K_0[\mathfrak{p}_3^r\bar{\mathfrak{p}}_3^s],T)$ with limit under the corestriction maps. Then the image of $\boldsymbol{\mathcal{Z}}^{(3)}_{\mu_3}$ in (\ref{eq:big-5}) under Shapiro's lemma is an Iwasawa cohomology class
\begin{equation}\label{eq:big-class}
\boldsymbol{\mathcal{Z}}_{\mu_3}\in H^1_{\rm Iw}\bigl(K_0[\mu_3 p^\infty],T_f^\vee(1-k/2)\otimes\tilde{\psi}_1^{-1}\tilde{\psi}_2^{-1}\tilde{\kappa}_{\rm ac,1}^{(k_1-2)/2}\tilde{\kappa}_{\rm ac,2}^{(k_2-2)/2}(1-(k_1+k_2)/2)\bigr)
\end{equation}
 for the conjugate self-dual representation $T_f^\vee(1-k/2)$ twisted by the Hecke character 
\[\chi^{-1}=\tilde{\psi}_1^{-1}\tilde{\psi}_2^{-1}\mathbf{N}^{1-(k_1+k_2)/2}\]
(up to an anticyclotomic twist).
Here $\chi$ is anticyclotomic and of infinity type (corresponding to the order $(\cP_1,\cP_2,\cP_3,\cP_4)$ or $(1,\tau_3,\tau_2,\tau_1)$):
\[\left(\frac{2-k_1-k_2}{2},\frac{k_1+k_2-2}{2},\frac{k_1-k_2}{2},\frac{k_2-k_1}{2}\right).\] 
Denote by
\begin{equation}\label{eq:self-rep}
T_{f,\chi}=T_f^{\vee}(1-k/2)\otimes\chi^{-1}.
\end{equation}

Following the proof of Theorem \ref{maintheorem1} and invoking \cite[Thm 6.4.1]{Rubin-ES}, we can obtain a collection of Iwasawa cohomology classes for anticyclotomic twists (to eliminate $\tilde{\kappa}_{\rm ac,1}^{(k_1-2)/2}\tilde{\kappa}_{\rm ac,2}^{(k_2-2)/2}$). We thus arrive at the proof for the wild norm relation, which is formulated inside the following theorem.

\begin{thm}\label{maintheorem2}
Suppose $p\nmid 6h_{K_0}$ and $f$ is non-Eisenstein modulo $\mathfrak{P}$. Let $\mu_3\in \mathcal{N}$ and denote $m=N_{K_3/\Q}(\mu_3)$. Then there exists a collection of Iwasawa cohomology classes
\[
\mathbf{z}_{f,\chi,\mu_3}\in H^1_{\rm Iw}\bigl(K_0[\mu_3 p^\infty],T_{f,\chi}\bigr)
\]
such that for every split prime $\la_3$ of $\cO_{K_3}$ of norm $\ell$, where $\ell$ splits in $K_0$, with $(\ell,mp)=1$ we have the norm relation
\begin{equation*}
    \mathrm{Norm}_{K_0[\mu_3]}^{K_0[\mu_3\la_3]}(\mathbf{z}_{f,\chi,\mu_3\la_3})=P_{\cL_4}(\mathrm{Frob}_{\cL_4})(\mathbf{z}_{f,\chi,\mu_3}),
\end{equation*}
where $P_{\cL_4}(X)=\det(1-X\cdot \mathrm{Frob}_{\cL_4}|\,(T_{f,\chi})^\vee(1))$.
\end{thm}

\section{Selmer groups }

In this section, we show that the classes constructed in Theorem \ref{maintheorem2} land in certain Selmer groups defined by Greenberg \cite{Greenberg55}. Keeping the setup at the start of Section~\ref{sec:main-thms}, we further assume that $f$ is a $p$-ordinary newform of even weight $k\geq 2$ with $p\nmid N_f$.

Let $\chi$ be an anticyclotomic Hecke character of $K_0$ of infinity type $(-a,a,-b,b)$ for some integers $a,b\geq 0$. We will focus on the conjugate self-dual $G_{K_0}$-representation 
\[
V_{f,\chi}:=V_f^\vee(1-k/2)\otimes\chi^{-1}.
\]

%To ease notation, set $T=T_{f,\psi_1,\psi_2}$ and put $V=T\otimes_{\Z_p}\Q_p$. 

\begin{defn}For each prime $\cP\in\{\cP_1,\cP_2,\cP_3,\cP_4\}$ of $K_0$ above $p$, we fix a $G_{K_{0,\cP}}$-stable subspace $\mathscr{F}_\cP^+(V_{f,\chi})\subset V_{f,\chi}$ and denote  
$$\mathscr{F}_\cP^-(V_{f,\chi})=V_{f,\chi}/\mathscr{F}_\cP^+(V_{f,\chi}).$$
Let $L$ be a finite extension of $K_0$. The \emph{Greenberg Selmer group} ${\rm Sel}_{\mathscr{F}}(L,V_{f,\chi})$ attached to $\mathscr{F}=\{\mathscr{F}_{\cP}^+(V_{f,\chi})\}_{\cP\vert p}$ is defined by 
\begin{equation}\label{eqn:Greenberg-Selmer-Defn}
    {\rm Sel}_{\mathscr{F}}(L,V_{f,\chi}):={\rm ker}\biggl\{H^1(L,V_{f,\chi})\rightarrow\prod_{w}\frac{H^1(L_w,V_{f,\chi})}{H^1_{\mathscr{F}}(L_w,V_{f,\chi})}
    %\times \prod_{w\nmid p}H^1(L_w^{\rm ur},V)
    \biggr\},
\end{equation}
where $w$ runs over the finite primes of $L$, and the local conditions are given by
\[
H^1_{\mathscr{F}}(L_w,V_{f,\chi})=\begin{cases}
{\rm ker}\bigl\{H^1(L_w,V_{f,\chi})\rightarrow H^1(L_{w}^{\rm ur},V_{f,\chi})\bigr\} & \textrm {if $w\nmid p$},  \\[0.2em]
{\rm ker}\bigl\{H^1(L_w,V_{f,\chi})\rightarrow H^1(L_{w},\mathscr{F}_{\cP}^-(V_{f,\chi}))\bigr\} & \textrm{if $w\mid \cP\mid p$}.   
\end{cases}
\]
We fix a lattice $T_{f,\chi}\subset V_{f,\chi}$. Let $H^1_{\mathscr{F}}(L_w,T_{f,\chi})$ be the inverse image of $H^1_{\mathscr{F}}(L_w,V_{f,\chi})$ under the natural map $$H^1(L_w,T_{f,\chi})\rightarrow H^1(L_w,V_{f,\chi}).$$ 
This then defines ${\rm Sel}_{\mathscr{F}}(L,T_{f,\chi})$ as in (\ref{eqn:Greenberg-Selmer-Defn}). For any $\Z_p^2$-extension $L_\infty=\bigcup_{r,s}L_{r,s}$ of $L$, we put
\[
{\rm Sel}_{\mathscr{F}}(L_\infty,T_{f,\chi}):=\varprojlim_{r,s}{\rm Sel}_{\mathscr{F}}(L_{r,s},T_{f,\chi}),
\]
where the inverse limit is taken with respect to the corestriction map. We also put ${\rm Sel}_{\mathscr{F}}(L_\infty,V_{f,\chi}):=
{\rm Sel}_{\mathscr{F}}(L_\infty,T_{f,\chi})\otimes_{\Z_p}\Q_p$. Note that this group is independent of the chosen lattice $T_{f,\chi}$.
\end{defn}

\begin{defn}
We also define the \emph{Bloch-Kato Selmer group} ${\rm Sel}_{\rm BK}(L,V_{f,\chi})$ following \cite{BK}:
\[
{\rm Sel}_{\rm BK}(L,V_{f,\chi}):={\rm ker}\biggl\{H^1(L,V_{f,\chi})\rightarrow\prod_w\frac{H^1(L_w,V_{f,\chi})}{H^1_f(L_w,V_{f,\chi})}\biggr\},
\]
where the local conditions are given by
\[
H^1_{f}(L_w,V_{f,\chi})=
{\rm ker}\bigl\{H^1(L_w,V_{f,\chi})\rightarrow H^1(L_{w}^{\rm ur},V_{f,\chi})\bigr\},\]
at primes $w\nmid p$, and the crystalline condition at primes $w\mid p$:
\[H^1_{f}(L_w,V_{f,\chi})={\rm ker}\bigl\{H^1(L_w,V_{f,\chi})\rightarrow H^1(L_{w},V_{f,\chi}\otimes\mathbf{B}_{\rm cris})\bigr\}  
\]
with $\mathbf{B}_{\rm cris}$ being Fontaine's crystalline period ring. 
The local conditions $H^1_f(L_w,T_{f,\chi})\subset H^1(L_w,T_{f,\chi})$ are defined by propagation similarly. 
\end{defn}

Besides the crystalline condition, there are three local conditions at primes $\cP\mid p$ that we will be interested in:
\begin{enumerate}
    \item The \textbf{strict} condition: $$\mathscr{F}_{\cP}^+(V_{f,\chi})=0$$
    \item The \textbf{relaxed} condition: $$\mathscr{F}_{\cP}^+(V_{f,\chi})=V_{f,\chi}$$
    \item The \textbf{ordinary} condition, corresponding to the fact that the restriction of $V_{f,\chi}$ to $G_{\Q_p}$ is reducible (see equation \ref{subsec:p-ord}): 
    \begin{equation*}
    \mathscr{F}_{\cP}^+(V_{f,\chi})=V_{f,\chi}^+:=V_f^{\vee,+}(1-k/2)\otimes\chi^{-1}
    \end{equation*}
\end{enumerate}

\begin{defn}\label{def:rel_str_ord}    Denote by ${\rm Sel}_{\alpha,\beta,\gamma,\delta}(K_0,V)$ the subgroup of $H^1(K_0,V)$ where classes are unramified at all primes $v\nmid p$; and they satisfy the conditions $\alpha$, $\beta$, $\gamma$, $\delta$ at $\cP_1$, $\cP_2$, $\cP_3$, $\cP_4$ respectively, where $\alpha,\beta,\gamma,\delta \in\{\text{rel},\text{str},\text{ord}\}$, and these conditions correspond to the relaxed, strict, and ordinary condition respectively.
\end{defn}

We will now compute the explicit local conditions for the Bloch-Kato Selmer group. Here we shall adopt the convention that the $p$-adic cyclotomic character has Hodge--Tate weight $-1$. Thus, since $\chi$ has infinity type $(-a,a,-b,b)$, the $p$-adic avatar of $\chi$ has Hodge--Tate weight $a,-a,b,-b$ at ${\cP_1,\cP_2,\cP_3,\cP_4}$ respectively.

\begin{lem}\label{lem:BK-Gr}
Assume that $a\ge b$. For any finite extension $L$ of $K_0$ we have
\begin{equation}
{\rm Sel}_{\rm BK}(L,V_{f,\chi})=\begin{cases}
{\rm Sel}_{\ord,\ord}(L,V_{f,\chi}) &\textrm{if $k\geq 2a+2$},\\[0.2em]
{\rm Sel}_{\relstr,\ord}(L,V_{f,\chi}) &\textrm{if $2b+2\leq k<2a+2$}\\[0.2em]
{\rm Sel}_{\relstr,\relstr}(L,V_{f,\chi}) &\textrm{if $k<2b+2$}.
\end{cases}\nonumber
\end{equation}
\end{lem}

\begin{proof}
%Recall that the $p$-adic cyclotomic character has Hodge--Tate weight $-1$ in our conventions. 
By the Panchiskin condition \cite[Thm 4.1(ii)]{BK} (see also \cite[(3.1)-(3.2)]{nekovarCRM} and \cite[Lem.\,2, p.\,125]{flach-CT}), for every prime $w\vert \cP\vert p$ of $L/K_0/\Q$  we have 
\[
H^1_f(L_w,V_{f,\chi})={\rm im}\bigl\{H^1(L_w,{\rm Fil}_{\cP}^1(V_{f,\chi}))\rightarrow H^1(L_w,V_{f,\chi})\bigr\},
\]
where ${\rm Fil}_{\cP}^1(V_{f,\chi})\subset V_{f,\chi}$  is a $G_{K_{\cP}}$-stable subspace (assuming it exists) such that the Hodge--Tate weights of ${\rm Fil}_{\cP}^1(V_{f,\chi})$ (resp. $V_{f,\chi}/{\rm Fil}_{\cP}^1(V_{f,\chi})$) are all $<0$  (resp. $\geq 0$).

Now, by computing the Hodge--Tate weights table of $V_{f,\chi}^+$ 
and $V_{f,\chi}^-:=V_{f,\chi}/V_{f,\chi}^{+}$ at the primes of $K_0$ above $p$:
\begin{center}
\begin{tabular}{c|c|c|}
\cline{2-3}
  & $V_{f,\chi}^+$ & $V_{f,\chi}^-$ \\
 \hline
\multicolumn{1}{|c|}{HT weight at $\cP_1$} & $-a-k/2$ & $-a-1+k/2$  \\
 \hline
 \multicolumn{1}{|c|}{HT weight at $\cP_2$} & $a-k/2$ & $a-1+k/2$  \\
 \hline
 \multicolumn{1}{|c|}{HT weight at $\cP_3$} & $-b-k/2$ & $-b-1+k/2$  \\
 \hline
 \multicolumn{1}{|c|}{HT weight at $\cP_4$} & $b-k/2$ & $b-1+k/2$  \\
 \hline
\end{tabular}
\end{center}
we obtained the equalities in the lemma.
\end{proof}

Fix a choice of Galois stable subgroups $\mathscr{F}=\{\mathscr{F}_{\cP}^+(V_{f,\chi})\}_{\cP\vert p}$ and let $$A_{f,\chi}:={\rm Hom}_{\Z_p}(T_{f,\chi},\mu_{p^\infty}).$$ %a finite extension $L/K$,  
Define the associated \emph{dual Selmer group} ${\rm Sel}_{\mathscr{F}^*}(L,A_{f,\chi})$ by
\[
{\rm Sel}_{\mathscr{F}^*}(L,A_{f,\chi}):={\rm ker}\biggl\{H^1(L,A_{f,\chi})\rightarrow\prod_{w}\frac{H^1(L_w,A_{f,\chi})}{H^1_{\mathscr{F}^*}(L_w,A_{f,\chi})}\biggr\},
\]
where $H^1_{\mathscr{F}^*}(L_w,A_{f,\chi})$ is the orthogonal complement of $H^1_{\mathscr{F}}(L_w,T_{f,\chi})$ under local Tate pairing
\[
H^1(L_w,T_{f,\chi})\times H^1(L_w,A_{f,\chi})\rightarrow\Q_p/\Z_p.
\]

One can then compute the following:
\begin{enumerate}
\item The dual Selmer group of ${\rm Sel}_{\relstr,\ord}(L,T_{f,\chi})$ consists of classes that are unramified outside $p$ and have the strict, relaxed, ordinary, ordinary condition at $\cP_1,\cP_2,\cP_3,\cP_4$ respectively. Compatibly with Definition \ref{def:rel_str_ord}, this can be denoted as ${\rm Sel}_{\strrel,\ord}(L,A_{f,\chi})$.
\item The dual Selmer group of ${\rm Sel}_{\relstr,\relstr}(L,T_{f,\chi})$ consists of classes that are unramified outside $p$ and have the strict, relaxed, strict, relaxed condition at $\cP_1,\cP_2,\cP_3,\cP_4$ respectively. Compatibly with Definition \ref{def:rel_str_ord}, this can be denoted as ${\rm Sel}_{\strrel,\strrel}(L,A_{f,\chi})$.
\item The dual Selmer group of ${\rm Sel}_{\ord,\ord}(L,T_{f,\chi})$ consists of classes that are unramified outside $p$, and land in the image of the natural map
\[
H^1(L_w,\mathscr{F}_{\cP}^+(A_{f,\chi}))\rightarrow H^1(L_w,A_{f,\chi}),\quad\textrm{$\mathscr{F}_{\cP}^+(A_{f,\chi}):={\rm Hom}_{\Z_p}(\mathscr{F}_{\cP}^-(T_{f,\chi}),\mu_{p^\infty})$,}
\]
for $w\vert {\cP}\vert p$. Compatibly with Definition \ref{def:rel_str_ord}, this can be denoted as ${\rm Sel}_{\ord,\ord}(L,A_{f,\chi})$.

\end{enumerate}

\section{Triple product \texorpdfstring{$p$}{p}-adic \texorpdfstring{$L$}{L}-function and Selmer group}\label{subsec:triple}

Here, we will recall some conventions on Hida families, triple product $p$-adic $L$-function ($\bff$-{unbalanced}) and Selmer groups (balanced and $\bff$-{unbalanced}) following \cite{hsieh-triple}.

\subsection{Hida families} 
\label{subsubsec:hida}
We follow the convention of \cite[\S{3.1}]{hsieh-triple}. Let $\cO$ be the ring of integers of a finite extension of $\Q_p$. Let $\cR$ be a normal domain, finite flat over the Iwasawa algebra
\[
\Lambda:=\cO\dBr{1+p\Z_p}.
\] 
 Let $N$ be a positive integer primes $p$ and $\chi:(\Z/Np\Z)^\times\rightarrow\cO^\times$ be a Dirichlet character. Denote by $S^o(N,\chi,\cR)\subset\cR\dBr{q}$ the space of ordinary $\cR$-adic cusp forms of tame level $N$ and branch character $\chi$.

Let $\mathfrak{X}_\cR^+\subset{\rm Spec}\,\cR(\overline{\Q}_p)$ be the set of {arithmetic points} of $\cR$, which consists of the ring homomorphisms $Q:\cR\rightarrow\overline{\Q}_p$ such that for some $k_Q\in\Z_{\geq 2}$ called the {weight of $Q$} and $\epsilon_Q(z)\in\mu_{p^\infty}$,
$$Q\vert_{1+p\Z_p}:z\mapsto z^{k_Q-1}\epsilon_Q(z).$$   We say that $\boldsymbol{f}=\sum_{n=1}^\infty a_n(\bfff)q^n\in S^o(N,\chi,\cR)$ is a {primitive Hida family} if the specialization $\boldsymbol{f}_Q$ for every $Q\in\mathfrak{X}_\cR^+$ gives the $q$-expansion of an ordinary $p$-stabilised newform of weight  $k_Q$ and tame conductor $N$. Let $\mathfrak{X}_{\cR}^{\rm cls}\subset{\rm Spec}\,\cR(\overline{\Q}_p)$ be the set of ring homomorphisms $Q$ as above with $k_Q\in\Z_{\geq 1}$ such that $\bfff_Q$ is the $q$-expansion of a classical modular form.

Given $\bfff$ a primitive Hida family of tame conductor $N$, one can associate a Galois representation 
\[
\rho_{\bfff}:G_\Q\rightarrow{\rm Aut}_{\cR}(V_\bfff)\simeq{\rm GL}_2(\cR),
\]
where the determinant of $\rho_\bfff$ is $\chi_\cR\cdot\varepsilon_{\rm cyc}$, see \cite[\S{3.2}]{hsieh-triple}. By \cite[Thm.~2.2.2]{wiles88}, the restriction of $V_\bfff$ to $G_{\Q_p}$ is reducible and one has a short exact sequence
\[
0\rightarrow V_\bfff^+\rightarrow V_\bfff\rightarrow V_\bfff^-\rightarrow 0.
\]
Here the quotient $V_\bfff^-$ is free of rank one over $\cR$, with $G_{\Q_p}$ acting via the unramified character sending an arithmetic Frobenius ${\rm Frob}_p^{-1}$ to $a_p(\bfff)$. Let $\bT(N,\cR)$ be the Hecke algebra acting on $\bigoplus_\chi S^o(N,\chi,\cR)$, where $\chi$ runs over the characters of $(\Z/Np\Z)^\times$. There is a $\cR$-algebra homomorphism attached to $\bfff$
\[
\lambda_{\bfff}:\bT(N,\cR)\rightarrow\cR
\] 
that factors through a local component $\bT_{\fkm}$, where $\fkm$ is the maximal ideal containing $\ker \lambda_{\bfff}$. Following \cite{hida-AJM}, we define the {congruence ideal} $C(\bfff)$ of $\bfff$ by
\[
C(\bfff):=\lambda_{\bfff}({\rm Ann}_{\bT_\fkm}({\rm ker}\,\lambda_\bfff))\subset\cR.
\] 
Under the assumption that the residual representation $\bar{\rho}_{\bfff}$ is absolutely irreducible and $p$-distinguished, Wiles \cite{Wiles} and Hida \cite{hida-AJM} prove that $C(\bfff)$ is generated by a nonzero element $\eta_{\bfff}\in\cR$.

\subsection{CM Hida families revisited}
\label{subsec:CM}
We explicitly construct CM Hida families, following the exposition in \cite[\S{8.1}]{hsieh-triple}. Let $K$ be an imaginary quadratic field of discriminant $-D_K<0$, and suppose that $p=\pp\ppbar$ splits in $K$, with $\pp$ the prime of $K$ above $p$ induced by our fixed embedding $\imath_p:\overline{\Q}\rightarrow\overline{\Q}_p$.

Let $K_\infty$ be the $\Z_p^2$-extension of $K$. Let $K({\pp^\infty})$ be the maximal subfield of $K_\infty$ unramified outside $\pp$. Put 
\[
\Gamma_\infty:={\rm Gal}(K_\infty/K)\simeq\Z_p^2,\quad\quad
\Gamma_{\pp}:={\rm Gal}(K({\pp^\infty})/K)\simeq\Z_p.
\]

For every ideal $\mathfrak{c}\subset\cO_K$, recall that $K_\mathfrak{c}$ is the ray class field of $K$ of conductor $\mathfrak{c}$. Using our notation, $K({\pp^\infty})$ is the maximal $\Z_p$-extension of $K$ inside $K_{\pp^\infty}$. Denote by ${\rm Art}_\pp$ the restriction of the Artin map to $K_\pp^\times$, with geometric normalisation. Then ${\rm Art}_\pp$ induces an embedding $1+p\Z_p\rightarrow\Gamma_{\pp}$, where we identified $\Z_p^\times$ and $\cO_{K_\pp}^\times$ via $\iota_p$. Let $\gamma_\pp$ be the image of $1+p$ hence it will be a topological generator of $\Gamma_{\pp}$.  %Write $I_\pp^{\rm w}={\rm Art}_\pp(1+p\Z_p)\vert_{K({\pp^\infty})}$ and put $[\Gamma_{\pp}:I_\pp^{\rm w}]=p^b$. If $(p,h_K)=1$ then we have $b=0$.

For each variable $S$ let $\Psi_S:\Gamma_\infty\rightarrow\cO\dBr{S}^\times$ be the universal character %(factoring through $\Gamma_{\pp^\infty}$) 
given by
\[
\Psi_S(\sigma)=(1+S)^{l(\sigma)},
\]
where $l(\sigma)\in\Z_p$ is such that $\sigma\vert_{K({\pp^\infty})}=\gamma_\pp^{l(\sigma)}$. %Upon possibly enlarging $\cO$, assume that it contains an element $\mathbf{v}$ with $\mathbf{v}^{p^b}=1+p$. 
Now assume that $\mathfrak{c}$ is prime to $p$. Given a finite order character $\ch:G_K\rightarrow\cO^\times$ of conductor dividing $\mathfrak{c}$, let
\[
\boldsymbol{\theta}_\ch(S)(q)=\sum_{(\fa,\pp\mathfrak{c})=1}\ch(\sigma_{\fa})\Psi^{-1}_{\frac{1+S}{1+p}-1}(\sigma_\fa)q^{N_{K/\Q}(\fa)}\in\cO\dBr{S}\dBr{q},
\]
where $\sigma_\fa\in{\rm Gal}(K_{\mathfrak{c}\pp^\infty}/K)$ is the Artin symbol of $\fa$. 
%With the conventions in \cite[p.\,435]{hsieh-triple}, 
Then $\boldsymbol{\theta}_\ch(S)$ is a Hida family  defined over $\cO\dBr{S}$ of tame level $N_{K/\Q}(\mathfrak{c})D_K$ and tame character $(\ch\circ\mathscr{V})\epsilon_{K}\omega^{-1}$, where $\mathscr{V}:G_\Q^{\rm ab}\rightarrow G_K^{\rm ab}$ is the transfer map and $\epsilon_{K}$ is the quadratic character corresponding to $K/\Q$.

\subsection{Triple products of Hida families}\label{subsubsec:triple-hida}

Let
\[
\bff\in S^o(N_f,\chi_f,\cR_f),\quad\bfg\in S^o(N_g,\chi_g,\cR_g),\quad\bfh\in S^o(N_h,\chi_h,\cR_h)
\]
be three primitive Hida families such that 
\begin{equation}\label{eq:a}
\textrm{$\chi_f\chi_g\chi_h=\omega^{2a}$ for some $a\in\Z$,}
\end{equation} 
where $\omega$ is the Teichm\"uller character. Let 
\[
\mathcal{R}=\cR_f\hat\otimes_{\cO}\cR_g\hat\otimes_{\cO}\cR_h
\] 
be a finite extension of the three-variable Iwasawa algebra $\Lambda\hat\otimes_{\cO}\Lambda\hat\otimes_{\cO}\Lambda$. 

Let $\mathfrak{X}_{\mathcal{R}}^+\subset{\rm Spec}\,\mathcal{R}(\overline{\Q}_p)$ be the weight space of $\mathcal{R}$ given by
\[
\mathfrak{X}_{\mathcal{R}}^+:=\left\{\underline{Q}=(Q_1,Q_2,Q_3)\in\mathfrak{X}_{\cR_f}^+\times\mathfrak{X}_{\cR_g}^{\rm cls}\times\mathfrak{X}_{\cR_h}^{\rm cls}\;:\;k_{Q_1}+k_{Q_2}+k_{Q_3}\equiv 0\;({\rm mod}\;2)\right\}.
\] 
One can then partition $\mathfrak{X}_\mathcal{R}^+=\mathfrak{X}_\mathcal{R}^{\rm bal}\sqcup\mathfrak{X}_\mathcal{R}^{\bff}\sqcup\mathfrak{X}_\mathcal{R}^{\bfg}\sqcup\mathfrak{X}_\mathcal{R}^{\bfh}$ as follows:
\begin{enumerate}
    \item the set of \emph{balanced} weights:
    \begin{align*}
\mathfrak{X}_{\mathcal{R}}^{\rm bal}&:=\left\{\underline{Q}\in\mathfrak{X}_{\mathcal{R}}^+\;:\;\textrm{$k_{Q_1}+k_{Q_2}+k_{Q_3}> 2k_{Q_i}$ for all $i\in\{1,2,3\}$}\right\},
\end{align*}
    \item the set of $\bff$-{unbalanced} weights:
    \begin{align*}
        \mathfrak{X}_{\mathcal{R}}^\bff&:=\left\{\underline{Q}\in\mathfrak{X}_{\mathcal{R}}^+\;:\;\textrm{$k_{Q_1}\geq k_{Q_2}+k_{Q_3}$}\right\},
    \end{align*}
    \item the set of $\bfg$-{unbalanced} weights:
    \begin{align*}
        \mathfrak{X}_{\mathcal{R}}^\bfg&:=\left\{\underline{Q}\in\mathfrak{X}_{\mathcal{R}}^+\;:\;\textrm{$k_{Q_2}\geq k_{Q_1}+k_{Q_3}$}\right\},
    \end{align*}
    \item the set of $\bfh$-{unbalanced} weights:
    \begin{align*}
    \mathfrak{X}_{\mathcal{R}}^\bfh&:=\left\{\underline{Q}\in\mathfrak{X}_{\mathcal{R}}^+\;:\;\textrm{$k_{Q_3}\geq k_{Q_1}+k_{Q_2}$}\right\}.
    \end{align*}
\end{enumerate}

Let $\mathbf{V}=V_\bff\hat\otimes_{\cO}V_{\bfg}\hat\otimes_{\cO}V_{\bfh}$ be the triple tensor product Galois representation attached to $(\bff,\bfg,\bfh)$. By (\ref{eq:a}), one can decompose the determinant of $\mathbf{V}$ as $\det\mathbf{V}=\mathcal{X}^2\varepsilon_{\rm cyc}$. Put  
\begin{equation}\label{eqn:V-dagger}
\Vdag:=\mathbf{V}\otimes\mathcal{X}^{-1}.
\end{equation}
This is a self-dual twist of $\mathbf{V}$. For any $\underline{Q}=(Q_1,Q_2,Q_3)\in\mathfrak{X}_{\mathcal{R}}^\bff$,  denote by $\VQdag$  the corresponding specialisation. 

\begin{comment}
We define 
\begin{equation}\label{eq:unb-intro}
\mathscr{F}_p^\bff(\Vdag):=V_{\bff}^+\hat\otimes_{\cO}V_\bfg\hat\otimes_{\cO}V_{\bfh}\otimes\mathcal{X}^{-1},
\end{equation}
which is a rank four $G_{\Q_p}$-invariant subspace $\mathscr{F}_p^\bff(\Vdag)\subset\Vdag$.

\end{comment}

For each prime $\ell$, let $\varepsilon_\ell(\VQdag)$ be the epsilon factor attached to the local representation $\VQdag\vert_{G_{\Q_\ell}}$  
(cf. \cite[p.\,21]{tate-background}). We assume that for some $\underline{Q}\in\mathfrak{X}_{\mathcal{R}}^\bff$, we have
\begin{equation}\label{eq:+1}
\textrm{ $\varepsilon_\ell(\VQdag)=+1$ for all prime factors $\ell$ of $N_f N_g N_h$.}
\end{equation}
Note that condition (\ref{eq:+1}) is independent of $\underline{Q}$ (see \cite[\S{1.2}]{hsieh-triple}). Furthermore it implies that the sign of the functional equation for the triple product $L$-function (with center at $s=0$) 
\[
L(\VQdag,s)
\] 
is $+1$ (resp. $-1$) for all $\underline{Q}\in\mathfrak{X}_{\mathcal{R}}^\bff\cup\mathfrak{X}_{\mathcal{R}}^\bfg\cup\mathfrak{X}_{\mathcal{R}}^\bfh$ (resp. $\underline{Q}\in\mathfrak{X}_{\mathcal{R}}^{\rm bal}$).

\begin{thm}[Theorem~A in \cite{hsieh-triple}]\label{thm:hsieh-triple}
Let $\bff,\bfg,\bfh$ be three primitive Hida families %as above 
satisfying conditions (\ref{eq:a}) and (\ref{eq:+1}). Assume also that $\gcd(N_f,N_g,N_h)$ is squarefree,
and the residual representation $\bar{\rho}_{\bff}$ is absolutely irreducible and $p$-distinguished.
Fix a generator $\eta_{\bff}$ of the congruence ideal of $\bff$. Then there exists a unique element 
\[
\mathscr{L}_p^{\bff,\eta_{\bff}}(\bff,\bfg,\bfh)\in\mathcal{R}
\]
such that for all $\underline{Q}=(Q_1,Q_2,Q_3)\in\mathfrak{X}_{\mathcal{R}}^{\bff}$ of weight $(k_1,k_2,k_3)$ with $\epsilon_{Q_1}=1$ we have
\[
\bigl(\mathscr{L}_p^{\bff,\eta_\bff}(\bff,\bfg,\bfh)(\underline{Q})\bigr)^2=\Gamma_{\VQdag}(0)\cdot\frac{L(\VQdag,0)}{(\sqrt{-1})^{2k_1}\cdot\Omega_{\bff_{Q_1}}^2}\cdot\mathcal{E}_p(\mathscr{F}_p^{\bff}(\VQdag))\cdot\prod_{q\in\Sigma_{\rm exc}}(1+q^{-1})^2,
\]
where:
\begin{itemize}
\item $\Gamma_{\VQdag}(0)=16(2\pi)^{-2k_1}\Gamma(w_{\underline{Q}})\Gamma(w_{\underline{Q}}+2-k_2-k_3)\Gamma(w_{\underline{Q}}+1-k_2)\Gamma(w_{\underline{Q}}+1-k_3)$, 
\[
\text{and  }\quad w_{\underline{Q}}=(k_1+k_2+k_3-2)/2;
\] 
\item $\Omega_{\bff_{Q_1}}$ is the Hida canonical period
\[
\Omega_{\bff_{Q_1}}:=(-2\sqrt{-1})^{k_1+1}\cdot\frac{\Vert\bff_{Q_1}^\circ\Vert_{\Gamma_0(N_f)}^2}{\eta_{\bff_{Q_1}}}\cdot\Bigl(1-\frac{\chi_{f}'(p)p^{k_1-1}}{\alpha_{Q_1}^2}\Bigr)\Bigl(1-\frac{\chi_{f}'(p)p^{k_1-2}}{\alpha_{Q_1}^2}\Bigr),
\]
with $\bff_{Q_1}^\circ\in S_{k_1}(\Gamma_0(N_{f}))$ the newform of conductor $N_f$ associated with $\bff_{Q_1}$, $\chi_f'$ the prime-to-$p$ part of $\chi_f$, and $\alpha_{Q_1}$ the specialisation of $a_p(\bff)\in\cR_f^\times$ at $Q_1$;
\item $\mathcal{E}_p(\mathscr{F}_p^{\bff}(\VQdag))$ is the modified $p$-Euler factor
and  $\Sigma_{\rm exc}$ is an explicitly defined subset of the prime factors of $N_f N_g N_h$, \cite[p.~416]{hsieh-triple}.
\begin{comment}
    \[
\mathcal{E}_p(\mathscr{F}_p^{\bff}(\VQdag)):=\frac{L_p(\mathscr{F}_p^{\bff}(\VQdag),0)}{\varepsilon_p(\mathscr{F}_p^{\bff}(\VQdag))\cdot L_p(\VQdag/\mathscr{F}_p^{\bff}(\VQdag),0)}\cdot\frac{1}{L_p(\VQdag,0)},
\]
\end{comment}

\end{itemize}

\end{thm}

\subsection{Triple product Selmer groups}
\label{subsec:tripleSelmer}

Recall from equation (\ref{eqn:V-dagger}) that $\Vdag=\mathbf{V}\otimes\mathcal{X}^{-1}$ is the self-dual twist of the Galois representation associated to a triple of primitive Hida families $(\bff,\bfg,\bfh)$ %as in $\S\ref{subsubsec:triple-hida}$
given (\ref{eq:a}).

\begin{defn}\label{def:local-p}
Let
\[
\mathscr{F}^{\rm bal}_p(\Vdag)%=\mathscr{F}_p^2(\Vdag)
:=\bigl(V_{\bff}\otimes V_{\bfg}^+\otimes V_{\bfh}^++V_{\bff}^+\otimes V_{\bfg}\otimes V_{\bfh}^++V_{\bff}^+\otimes V_{\bfg}^+\otimes V_{\bfh}\bigr)\otimes\mathcal{X}^{-1},
\]
and define the {balanced local condition} $\rH^1_{\rm bal}(\Q_p,\Vdag)$ by 
\[
\rH^1_{\rm bal}(\Q_p,\Vdag):={\rm im}\bigl(\rH^1(\Q_p,\mathscr{F}_p^{\rm bal}(\Vdag))\rightarrow\rH^1(\Q_p,\Vdag)\bigr).
\]
Similarly, let 
\begin{align*}
    \mathscr{F}_p^{\unb}(\Vdag)=\bigl(V_{\bff}^+\otimes V_{\bfg}\otimes V_{\bfh}\bigr)\otimes\mathcal{X}^{-1},\\
    \mathscr{F}_p^{\bfh}(\Vdag)=\bigl(V_{\bff}\otimes V_{\bfg}\otimes V_{\bfh}^+\bigr)\otimes\mathcal{X}^{-1}.
\end{align*}
 Define the {$\bff$-unbalanced local condition} $\rH^1_{\unb}(\Q_p,\Vdag)$ by
\[
\rH^1_{\bff}(\Q_p,\Vdag):={\rm im}\bigl(\rH^1(\Q_p,\mathscr{F}_p^{\bff}(\Vdag))\rightarrow\rH^1(\Q_p,\Vdag)\bigr)
\]
and the {$\bfh$-unbalanced local condition} $\rH^1_{\bfh}(\Q_p,\Vdag)$ by
\[
\rH^1_{\bfh}(\Q_p,\Vdag):={\rm im}\bigl(\rH^1(\Q_p,\mathscr{F}_p^{\bfh}(\Vdag))\rightarrow\rH^1(\Q_p,\Vdag)\bigr).
\]
\end{defn}

Note that the maps appearing in these definitions are injective, so we can identify $\rH^1_{\star}(\Q_p,\Vdag)$ with $\rH^1(\Q_p,\mathscr{F}_p^\star(\Vdag))$ for $\star\in\{{\rm bal},\bff,\bfh\}$.

\begin{defn}\label{def:Sel-bu}
Let $\star\in\{{\rm bal},\bff,\bfh\}$. Define the Selmer group ${\rm Sel}^\star(\Q,\Vdag)$ by
\[
{\rm Sel}^\star(\Q,\Vdag):=\ker\biggl\{\rH^1(\Q,\Vdag)\rightarrow\frac{\rH^1(\Q_p,\Vdag)}{\rH^1_\star(\Q_p,\Vdag)}\times\prod_{v\neq p}\rH^1(\Q_v^{\rm nr},\Vdag)\biggr\}.
\]
We call ${\rm Sel}^{\rm bal}(\Q,\Vdag)$ the {balanced} Selmer group, ${\rm Sel}^{\bff}(\Q,\Vdag)$ the {$\bff$-unbalanced} Selmer group, and ${\rm Sel}^{\bfh}(\Q,\Vdag)$ the {$\bfh$-unbalanced} Selmer group.
\end{defn}
\begin{defn}
Let $\Adag={\rm Hom}_{\Z_p}(\Vdag,\mu_{p^\infty})$ and let $\star\in\{{\rm bal},\bff,\bfh\}$. Define $\rH^1_{\star}(\Q_p,\Adag)\subset\rH^1(\Q_p,\Adag)$ to be the orthogonal complement of $\rH^1_\star(\Q_p,\Vdag)$ under the local Tate pairing
\[
\rH^1(\Q_p,\Vdag)\times\rH^1(\Q_p,\Adag)\rightarrow\Q_p/\Z_p.
\]
Similarly as above, we then define the balanced, the $\bff$-unbalanced and the $\bfh$-unbalanced  Selmer groups with coefficients in $\Adag$ by
\[
{\rm Sel}^\star(\Q,\Adag):=\ker\biggl\{\rH^1(\Q,\Adag)\rightarrow\frac{\rH^1(\Q_p,\Adag)}{\rH_\star^1(\Q_p,\Adag)}\times\prod_{v\neq p}\rH^1(\Q_v^{\rm nr},\Adag)\biggr\}.
\]
Let $$X^\star(\Q,\Adag)={\rm Hom}_{\Z_p}({\rm Sel}^\star(\Q,\Adag),\Q_p/\Z_p)$$ be the Pontryagin dual of ${\rm Sel}^\star(\Q,\Adag)$.
 
\end{defn}

\section{Arithmetic applications}
Finally, we obtain some arithmetic results from our constructed Euler system through identifying our classes as an anticyclotomic Euler system in the sense of Jetchev--Nekov\'a{\v r}--Skinner \cite{JNS} and the explicit reciprocity law.

\subsection{Reciprocity law and Greenberg--Iwasawa main conjectures}
\label{subsec:diag}

Let $(\bff,\bfg,\bfh)$ be a triple of primitive Hida families as in $\S\ref{subsubsec:hida}$ satisfying (\ref{eq:a}). Let $N={\rm lcm}(N_f,N_g,N_h)$. The big diagonal class constructed in \cite[\S{8.1}]{BSV}
\begin{equation}\label{eq:3-diag}
\kappa(\bff,\bfg,\bfh)\in\rH^1(\Q,\Vdag(N)),
\end{equation}
where $\Vdag(N)$ (this is $V(\bff,\bfg,\bfh)$ using the notation of \cite{BSV}) is a free $\mathcal{R}$-module isomorphic to finitely many copies of $\Vdag$, can be identified with classes $\widetilde{\boldsymbol{\kappa}}_m^{(1)}$, $\boldsymbol{\kappa}_m^{(2)}$ in equation $(2.14)$, $(2.15)$ respectively of \cite{CD}. The definition of the Selmer groups in $\S\ref{subsec:tripleSelmer}$ extends to $\Vdag(N)$, and by \cite[Cor.\,8.2]{BSV} we have $\kappa(\bff,\bfg,\bfh)\in{\rm Sel}^{\rm bal}(\Q,\Vdag(N))$. Now we choose level-$N$ test vectors $(\breve{\bff},\breve{\bfg},\breve{\bfh})$, provided by \cite[Thm.~A]{hsieh-triple}, to project the classes from $\mathbb{V}^\dagger(N)$ to $\mathbb{V}^\dagger$.

We define more $G_{\Q_p}$-invariant subspaces of $\Vdag$:  
\begin{equation}\label{eq:3-summands}
\begin{aligned}
\mathscr{F}_p^{3}(\Vdag)&=V_\bff^+\hat\otimes_{\cO} V_{\bfg}^+\hat\otimes_{\cO} V_{\bfh}^+\otimes\mathcal{X}^{-1}, \\
\mathbf{V}_{\bff}^{\bfg\bfh}&=V_\bff^-\hat\otimes_{\cO} V_{\bfg}^+\hat\otimes_{\cO} V_{\bfh}^+\otimes\mathcal{X}^{-1},\\
\mathbf{V}_{\bfg}^{\bff\bfh}&=V_\bff^+\hat\otimes_{\cO} V_{\bfg}^-\hat\otimes_{\cO} V_{\bfh}^+\otimes\mathcal{X}^{-1},\\
\mathbf{V}_{\bfh}^{\bff\bfg}&=V_\bff^+\hat\otimes_{\cO} V_{\bfg}^+\hat\otimes_{\cO} V_{\bfh}^-\otimes\mathcal{X}^{-1},
\end{aligned}
\end{equation}
%and denote by $\mathscr{F}_p^{3}(\Vdag)\subset\Vdag$ the corresponding restriction to the line $S_2=S_3$. 
and obtain 
\begin{equation}\label{eq:gr2}
\mathscr{F}^{\rm bal}_p(\Vdag)/\mathscr{F}_p^3(\Vdag)\cong
\mathbf{V}_{\bff}^{\bfg\bfh}\oplus\mathbf{V}_{\bfg}^{\bff\bfh}\oplus\mathbf{V}_{\bfh}^{\bff\bfg},
\end{equation}

Assume that the congruence ideal $C(\bff)\subset\cR_f$ is principal, generated by the nonzero $\eta_\bff\in\cR_f$ (this will be satisfied when the residual representation $\bar{\rho}_{\bff}$ is absolutely irreducible and $p$-distinguished). 
One can deduce from results in \cite{KLZ} the construction of an injective three-variable $p$-adic regulator map with pseudo-null cokernel:
\begin{equation}\label{eq:Log}
{\rm Log}^{\eta_\bff}
:\rH^1(\Q_p,\mathbf{V}_{\bff}^{\bfg\bfh})\rightarrow\mathcal{R},
\end{equation}
see the explicit map in \cite[\S{4.3.1}]{CD} and the explanation in \cite[\S{7.3}]{BSV}.
\begin{comment}
    characterised by the property that for all $\mathfrak{Z}\in\rH^1(\Q_p,\mathbf{V}_{\bff}^{\bfg\bfh})$ and all points $\underline{Q}=(Q_1,Q_2,Q_3)\in\mathfrak{X}_{\mathcal{R}}^{\bff}$ of weight $(k_0,k_1,k_2)$ with $\epsilon_{Q_i}=1$ ($i=0,1,2$) we have
\begin{align*}
\frac{{\rm Log}^{\eta_\bff}(\mathfrak{Z})(\underline{Q})}{\eta_{\bff_{Q_1}}}&=(p-1)\alpha_{Q_1}\biggl(1-\frac{\beta_{Q_1} \alpha_{Q_2} \alpha_{Q_3}}{p^{w_{\underline{Q}}}}\biggr)\biggl(1-\frac{\alpha_{Q_1} \beta_{Q_2}\beta_{Q_3}}{p^{w_{\underline{Q}}}}\biggr)^{-1}\\
&\quad\times\begin{cases}
\frac{(-1)^{w_{\underline{Q}}-k_0}}{(w_{\underline{Q}}-k_0)!}\cdot\left\langle{\rm log}_p(\mathfrak{Z}_{\underline{Q}}),\eta_{\bff_{Q_1}}\otimes\omega_{\bfg_{Q_2}}\otimes\omega_{\bfh_{Q_3}}\right\rangle_{\rm dR}, &\textrm{if $\underline{Q}\in\mathfrak{X}_{\mathcal{R}}^{\rm bal}$,}\\[0.5em]
(k_0-w_{\underline{Q}}-1)!\cdot\left\langle{\rm exp}_p^*(\mathfrak{Z}_{\underline{Q}}),\eta_{\bff_{Q_1}}\otimes\omega_{\bfg_{Q_2}}\otimes\omega_{\bfh_{Q_3}}\right\rangle_{\rm dR},&\textrm{if $\underline{Q}\in\mathfrak{X}_{\mathcal{R}}^{\bff}$.}
\end{cases}
\end{align*}

Here, $w_{\underline{Q}}=(k_0+k_1+k_2-2)/2$ is as in Theorem~\ref{thm:hsieh-triple}, $\alpha_{Q_1}$ denotes the specialisation of $a_p(\bff)$ at $Q_1$, we put $\beta_{Q_1}=\chi_f'(p)p^{k_0-1}\alpha_{Q_1}^{-1}$, and $(\alpha_{Q_2},\beta_{Q_2})$ (resp. $(\alpha_{Q_3},\beta_{Q_3})$) are defined likewise with  $\bfg$ (resp. $\bfh$) in place of $\bff$.

\end{comment}

Let ${\rm res}_p(\kappa(\bff,\bfg,\bfh))_{\bff}$ be the image of $\kappa(\bff,\bfg,\bfh)$ under the natural composition of maps:
\begin{equation}\label{eq:map-ERL}
{\rm Sel}^{\rm bal}(\Q,\Vdag)\xrightarrow{{\rm res}_p}\rH^1(\Q_p,\mathscr{F}_p^{\rm bal}(\Vdag))\rightarrow\rH^1(\Q_p,\mathscr{F}_p^{\rm bal}(\Vdag)/\mathscr{F}_p^3(\Vdag))\rightarrow\rH^1(\Q_p,\mathbf{V}_{\bff}^{\bfg\bfh}),
\end{equation}
where we first restrict at $p$ and then project onto the first direct summand in (\ref{eq:gr2}). The following result is an explicit reciprocity law that relates diagonal cycles with the triple product $p$-adic $L$-functions.

\begin{thm}[Theorem~A in \cite{BSV}]
\label{thm:ERL}
Let $(\bff,\bfg,\bfh)$ be a triple of primitive Hida families as in Theorem~\ref{thm:hsieh-triple}. Then
\[
{\rm Log}^{\eta_\bff}({\rm res}_p(\kappa(\bff,\bfg,\bfh))_\bff)=\mathscr{L}_p^{\bff,\eta_\bff}(\bff,\bfg,\bfh).
\]
\end{thm}

Assume that the associated ring $\mathcal{R}$ is  regular. Similar to \cite[\S{7.3}]{ACR}, the following result can be seen as the equivalence between two different formulations of the Iwasawa main conjecture in the style of Greenberg \cite{Greenberg55} for the $p$-adic deformation $\Vdag$, relating the $\bff-$unbalanced Selmer group to the balanced one (or one with $\mathscr{L}_p^{\bff}$ and another `without $p$-adic $L$-functions'). 

%(see \cite[\S{7.3}]{ACR} for further details).

\begin{prop}[Proposition 4.3.3 in\cite{CD}]\label{prop:equiv}
The following statements {\rm (I)} and {\rm (II)} are equivalent:
\begin{enumerate}
\item[(I)] $\mathscr{L}_p^{\unb,\eta_{\bff}}(\bff,\bfg,\bfh)$ is nonzero, the modules ${\rm Sel}^{\unb}(\Q,\Vdag)$ and $X^{\unb}(\Q,\Adag)$ are both $\mathcal{R}$-torsion, and 
\[
{\rm char}_\mathcal{R}\bigl(X^{\unb}(\Q,\Adag)\bigr)=\bigl(\mathscr{L}_p^\unb(\bff,\bfg,\bfh)^2\bigr)
\]
in $\mathcal{R}\otimes_{\Z_p}\Q_p$.

\item[(II)] $\kappa(\bff,\bfg,\bfh)$ is not $\mathcal{R}$-torsion, the modules ${\rm Sel}^{\rm bal}(\Q,\Vdag)$ and $X^{\rm bal}(\Q,\Adag)$ have both $\mathcal{R}$-rank one, and
\[
{\rm char}_{\mathcal{R}}\bigl(X^{\rm bal}(\Q,\Adag)_{\rm tors}\bigr)={\rm char}_{\mathcal{R}}\biggl(\frac{{\rm Sel}^{\rm bal}(\Q,\Vdag)}{\mathcal{R}\cdot\kappa(\bff,\bfg,\bfh)}\biggr)^2
\]
in $\mathcal{R}\otimes_{\Z_p}\Q_p$, where the subscript ${\rm tors}$ denotes the $\mathcal{R}$-torsion submodule.
\end{enumerate}
\end{prop}

\subsection{The set-up}\label{subsec:factor-L-def+Sel}

Let $f\in S_{2r}(pN_f)$ be a $p$-stabilised newform, and suppose the residual representation $\bar\rho_f$ is absolutely irreducible and $p$-distinguished. By Hida theory, $f$ is the specialisation of a unique primitive Hida family $\bff\in S^o(N_f,\cR)$ 
%with branch character $\chi_{f}=\mathds{1}$ 
at an arithmetic point $Q_1\in\mathfrak{X}_\cR^+$ of weight $2r$.  
For $i\in\{1,2\}$ let $\fkf_i\subset\cO_{K_i}$ be an ideal coprime to $pN_f$, $\ch_i$ be ray class characters of $K_i$ of conductors dividing $\mathfrak{f}_i$. Let $\chi_{\ch_i}$ be the central character of $\ch_i$. We assume that  
\begin{equation}\label{eq:sd-chars}
\chi_{\ch_1}\epsilon_{K_1}\chi_{\ch_2}\epsilon_{K_2}=1,
\end{equation}
and let
\begin{equation}\label{eq:gg*}
\bfg_1=\boldsymbol{\theta}_{\ch_1}(S_1)\in\cO\dBr{S_1}\dBr{q}, 
\quad
\bfg_2=\boldsymbol{\theta}_{\ch_2}(S_2)\in\cO\dBr{S_2}\dBr{q}
\end{equation}
be the CM Hida families %(of tame level $C=D_K\ell^{2m}$) 
attached to $\ch_1$ and $\ch_2$, respectively.

The triple $(\bff,\bfg_1,\bfg_2)$ satisfies conditions (\ref{eq:a}) and the associated $\bff$-unbalanced triple product $p$-adic $L$-function $\mathscr{L}_p^{\bff,\eta_{\bff}}(\bff,\bfg_1,\bfg_2)$ %of Theorem~\ref{thm:hsieh-triple} 
is an element in $\mathcal{R}=\cR\hat\otimes_{\cO}\cO\dBr{S_1}\hat\otimes_{\cO}\cO\dBr{S_2}\simeq\cR\dBr{S_1,S_2}$. Let 
\begin{equation}\label{eq:triple-2var}
\mathscr{L}_p^{\unb,\eta_\bff}(f,\bfg_1,\bfg_2)\in\cO\dBr{S_1,S_2}
\end{equation}
be its image  
%be the image of $\mathscr{L}_p^{\unb,\eta_\bff}(\bff,\bfg_1,\bfg_2)$ 
under the natural map $\cR\dBr{S_1,S_2}\rightarrow\cO\dBr{S_1,S_2}$ defined by $Q_1$.

Write $\Vsdag$ for the specialisation of $\Vdag$ at $Q_1$. Let $V_f^\vee$ be the Galois representation associated to $f$, and recall that $\det(V_f^\vee)=\varepsilon_{\rm cyc}^{2r-1}$ in our conventions. Setting $T_i=\mathbf{v}^{-1}(1+S_i)-1$ ($i\in\{1,2\}$), we have 
$\det(V_{\bfg_{T_1}}\otimes V_{\bfh_{T_2}})=\Psi_{T_1}\Psi_{T_2}\circ\mathscr{V}$, and so
\begin{equation}\label{eq:dec-V-first}
\begin{aligned}
\Vsdag&\simeq T_f^{\vee}\otimes({\rm Ind}_K^\Q\ch_1^{-1}\Psi_{T_1})\otimes({\rm Ind}_K^\Q\ch_2^{-1}\Psi_{T_2})\otimes\varepsilon_{\rm cyc}^{1-r}(\Psi_{T_1}^{-1/2}\Psi_{T_2}^{-1/2}\circ\mathscr{V})\\
&\simeq T_f^{\vee}(1-r)\otimes{\rm Ind}_{K_0}^\Q\tilde{\ch}_1^{-1}\tilde{\Psi}_{T_1}\tilde{\ch}_2^{-1}\tilde{\Psi}_{T_2},
\end{aligned}
\end{equation}
where $T_f^{\vee}$ is a $G_\Q$-stable $\cO$-lattice inside $V_f^{\vee}$. In particular, we get
\begin{equation}\label{eq:shapiro}
\rH^1(\Q,\Vsdag)\simeq\rH^1(K_0,T_f^{\vee}(1-r)\otimes\tilde{\ch}_1^{-1}\tilde{\Psi}_{T_1}\tilde{\ch}_2^{-1}\tilde{\Psi}_{T_2})
\end{equation}
by Shapiro's lemma.

\subsection{Local conditions at $p$ of the Euler system}\label{subsec:local-condition} 

%Recall that $\psi_i$ is a Hecke character of $K_i$ of infinity type $(1-k_i,0)$ for $i=1,2$ and $k_i\geq 1$,  whose central characters satisfy $\chi_{\psi_1}\epsilon_{K_1}\chi_{\psi_2}\epsilon_{K_2}=1$. 

Recall from Theorem~\ref{maintheorem2} that we have constructed classes
\[
\mathbf{z}_{f,\chi,\mu_3}\in H^1_{\rm Iw}(K_0[\mu_3p^\infty],T_{f,\chi}),
\]
where $T_{f,\chi}=T_f^{\vee}(1-r)\otimes\chi^{-1}$ and $\chi^{-1}=\tilde{\psi}_1^{-1}\tilde{\psi}_2^{-1}\mathbf{N}^{1-(k_1+k_2)/2}$.  

\begin{prop}
\label{propselmer2}
Suppose $p\nmid 6h_{K_0}$ and $f$ is non-Eisenstein modulo $\mathfrak{P}$. Let $\mu_3\in\mathcal{N}$ (taken from Section \ref{subsec:prime-decomp}) run over squarefree product of prime ideals of $\la_3\in\mathcal{L}$  with $m=N_{K_3/\Q}(\mu_3)$ coprime to $p$. The class $\mathbf{z}_{f,\chi,\mu_3}$ of Theorem\,\ref{maintheorem2} satisfies
\[
\mathbf{z}_{f,\chi,\mu_3}\in{\rm Sel}_{\relstr,\ord}(K_0[\mu_3p^\infty],T_{f,\chi}).
\]
\end{prop}

\begin{proof}
By \cite[Cor.\,8.2]{BSV} and \cite[Sec.\,3.2]{CD}, the class $\mathbf{z}_{f,\chi,\mu_3}$ lands in the balanced Selmer group ${\rm Sel}^{\rm bal}(\Q,\Vdag)$, where the balanced local condition at $p$ upon specialised to $f$ is given by
\begin{equation}\label{eq:bal-def}
\begin{aligned}
\mathscr{F}^{\rm bal}_p(\Vsdag)
\simeq\bigl(T_f^{\vee}(1-r)\otimes\tilde{\ch}_1^{-1}\tilde{\Psi}_{T_1}\tilde{\ch}_2^{-1}\tilde{\Psi}_{T_2}\bigr)&\oplus\bigl(T_f^{\vee,+}(1-r)\otimes\tilde{\ch}_1^{-1}\tilde{\Psi}_{T_1}\tilde{\ch}_2^{-\cc}\tilde{\Psi}_{T_2}^{\cc}\bigr)\\
&\oplus\bigl(T_f^{\vee,+}(1-r)\otimes\tilde{\ch}_1^{-\cc}\tilde{\Psi}_{T_1}^{\cc}\tilde{\ch}_2^{-1}\tilde{\Psi}_{T_2}\bigr).
\end{aligned}
\end{equation}

Put $\widetilde{\mathbf{V}}_{Q_1}^\dagger=T_f^{\vee}(1-r)\otimes\tilde{\ch}_1^{-1}\tilde{\Psi}_{T_1}\tilde{\ch}_2^{-1}\tilde{\Psi}_{T_2}$ then Shapiro Lemma tells us that $H^1(\Q,\Vsdag)\simeq H^1(K_0,\widetilde{\mathbf{V}}_{Q_1}^\dagger)$, see (\ref{eq:shapiro}). Following \cite[Sec.\,5.3]{CD}, the local condition  $\mathscr{F}_p^{\rm bal}(\Vsdag)$ cutting out the specialised balanced Selmer group at $p$ corresponds to
\begin{equation}
\begin{split}
\mathscr{F}_{\cP_1}^{\rm bal}(\widetilde{\mathbf{V}}^\dagger_{Q_1}\vert_{G_{K_0}})&=T_f^\vee(1-k/2)\otimes\tilde{\ch}_1^{-1}\tilde{\Psi}_{T_1}\tilde{\ch}_2^{-1}\tilde{\Psi}_{T_2},\\
\mathscr{F}_{\cP_2}^{\rm bal}(\widetilde{\mathbf{V}}^\dagger_{Q_1}\vert_{G_{K_0}})&=0,\\
\mathscr{F}_{\cP_3}^{\rm bal}(\widetilde{\mathbf{V}}^\dagger_{Q_1}\vert_{G_{K_0}})&=T_f^{\vee,+}(1-k/2)\otimes\tilde{\ch}_1^{-\cc}\tilde{\Psi}_{T_1}^{\cc}\tilde{\ch}_2^{-1}\tilde{\Psi}_{T_2},\\
\mathscr{F}_{\cP_4}^{\rm bal}(\widetilde{\mathbf{V}}^\dagger_{Q_1}\vert_{G_{K_0}})&=T_f^{\vee,+}(1-k/2)\otimes\tilde{\ch}_1^{-1}\tilde{\Psi}_{T_1}\tilde{\ch}_2^{-\cc}\tilde{\Psi}_{T_2}^{\cc}.
\end{split}
\end{equation}
Hence the class $\mathbf{z}_{f,\chi,\mu_3}$ satisfies the relaxed-strict-ordinary-ordinary condition at the primes above $p$.

On the other hand, at the primes $w\nmid p$, because $V_{f,\chi}$ is conjugate self-dual and pure of weight $-1$, we see that 
\[
H^0(K_0[\mu_3\pp_3^r\ppbar_3^s]_w,V_{f,\chi})=H^2(K_0[\mu_3\pp_3^r\ppbar_3^s]_w,V_{f,\chi})=0
\]
for all $r,s$, and therefore $$H^1(K_0[\mu_3\pp_3^r\ppbar_3^s]_w,V_{f,\chi})=0$$ by Tate's local Euler characteristic formula. This implies the torsionness of  $H^1(K_0[\mu_3\pp_3^r\ppbar_3^s]_w,T_{f,\chi})$, and one has the following inclusion:
\[
{\rm res}_w(\mathbf{z}_{f,\chi,\mu_3})\in \varprojlim_{r,s} H^1_f(K_0[\mu_3\pp_3^r\ppbar_3^s]_w,T_{f,\chi}),
\]
which concludes the proof.
\end{proof}

\begin{prop}\label{prop:factor-S}
Via the isomorphism (\ref{eq:shapiro}), 
\begin{enumerate}
    \item the balanced Selmer group ${\rm Sel}^{\rm bal}(\Q,\Vsdag)$ can be rewritten as
\begin{align*}
{\rm Sel}^{\rm bal}(\Q,\Vsdag)&\simeq{\rm Sel}_{\relstr,\ord}(K_0,T_f^{\vee}(1-r)\otimes\tilde{\ch}_1^{-1}\tilde{\Psi}_{T_1}\tilde{\ch}_2^{-1}\tilde{\Psi}_{T_2}),
\end{align*}
\item the $\unb$-unbalanced Selmer group ${\rm Sel}^{\unb}(\Q,\Vsdag)$ can be rewritten as
\begin{align*}
{\rm Sel}^{\unb}(\Q,\Vsdag)&\simeq
{\rm Sel}_{\ord,\ord}(K_0,T_f^{\vee}(1-r)\otimes\tilde{\ch}_1^{-1}\tilde{\Psi}_{T_1}\tilde{\ch}_2^{-1}\tilde{\Psi}_{T_2}).\end{align*}
\item the $\bfh$-unbalanced Selmer group ${\rm Sel}^{\bfh}(\Q,\Vsdag)$ can be rewritten as
\begin{align*}
{\rm Sel}^{\bfh}(\Q,\Vsdag)&\simeq
{\rm Sel}_{\relstr,\relstr}(K_0,T_f^{\vee}(1-r)\otimes\tilde{\ch}_1^{-1}\tilde{\Psi}_{T_1}\tilde{\ch}_2^{-1}\tilde{\Psi}_{T_2}).\end{align*}
\end{enumerate}
\end{prop}

\begin{proof} For the balanced case, see Proposition \ref{propselmer2}. For the $\unb$-unbalanced case, note that 
\begin{equation}\label{eq:unbal-def}
\begin{aligned}
\mathscr{F}^{\bff}_p(\Vsdag)
\simeq\bigl(T_f^{\vee,+}(1-r)\otimes\tilde{\ch}_1^{-1}\tilde{\Psi}_{T_1}\tilde{\ch}_2^{-1}\tilde{\Psi}_{T_2}\bigr)&\oplus\bigl(T_f^{\vee,+}(1-r)\otimes\tilde{\ch}_1^{-1}\tilde{\Psi}_{T_1}\tilde{\ch}_2^{-\cc}\tilde{\Psi}_{T_2}^{\cc}\bigr)\\
\oplus\bigl(T_f^{\vee,+}(1-r)\otimes\tilde{\ch}_1^{-\cc}\tilde{\Psi}_{T_1}^{\cc}\tilde{\ch}_2^{-1}\tilde{\Psi}_{T_2}\bigr)&\oplus\bigl(T_f^{\vee,+}(1-r)\otimes\tilde{\ch}_1^{-\cc}\tilde{\Psi}_{T_1}^{\cc}\tilde{\ch}_2^{-\cc}\tilde{\Psi}_{T_2}^{\cc}\bigr).
\end{aligned}
\end{equation}
and the result follows. 

The $\bfh$-unbalanced case can be obtained in a similar manner where:
\begin{equation*}
    \begin{aligned}
\mathscr{F}^{\bfh}_p(\Vsdag)
\simeq\bigl(T_f^{\vee}(1-r)\otimes\tilde{\ch}_1^{-1}\tilde{\Psi}_{T_1}\tilde{\ch}_2^{-1}\tilde{\Psi}_{T_2}\bigr)
\oplus\bigl(T_f^{\vee}(1-r)\otimes\tilde{\ch}_1^{-\cc}\tilde{\Psi}_{T_1}^{\cc}\tilde{\ch}_2^{-1}\tilde{\Psi}_{T_2}\bigr).
\end{aligned}
\end{equation*}
\end{proof}

As a consequence, we also obtain the following isomorphisms for the Selmer groups with coefficients in $\mathbf{A}_{Q_1}^\dagger={\rm Hom}_{\Z_p}(\Vsdag,\mu_{p^\infty})$ by local Tate duality. Let $A_f(r)={\rm Hom}_{\Z_p}(T_f^\vee(1-r),\mu_{p^\infty})$.

\begin{cor}\label{cor:factor-S} We can identify
the balanced Selmer group ${\rm Sel}^{\rm bal}(\Q,\mathbf{A}_{Q_1}^\dagger)$ as
\begin{align*}
{\rm Sel}^{\rm bal}(\Q,\mathbf{A}_{Q_1}^\dagger)&\simeq{\rm Sel}_{\strrel,\ord}(K_0,A_f(r)\otimes\tilde{\ch}_1\tilde{\Psi}_{T_1}^{-1}\tilde{\ch}_2\tilde{\Psi}_{T_2}^{-1}),
\end{align*}
the $\unb$-unbalanced Selmer group ${\rm Sel}^{\unb}(\Q,\mathbf{A}_{Q_1}^\dagger)$ as
\begin{align*}
{\rm Sel}^{\unb}(\Q,\mathbf{A}_{Q_1}^\dagger)&\simeq
{\rm Sel}_{\ord,\ord}(K_0,A_f(r)\otimes\tilde{\ch}_1\tilde{\Psi}_{T_1}^{-1}\tilde{\ch}_2\tilde{\Psi}_{T_2}^{-1}),
\end{align*}
and the $\bfh$-unbalanced Selmer group ${\rm Sel}^{\bfh}(\Q,\mathbf{A}_{Q_1}^\dagger)$ as
\begin{align*}
{\rm Sel}^{\bfh}(\Q,\mathbf{A}_{Q_1}^\dagger)&\simeq
{\rm Sel}_{\strrel,\strrel}(K_0,A_f(r)\otimes\tilde{\ch}_1\tilde{\Psi}_{T_1}^{-1}\tilde{\ch}_2\tilde{\Psi}_{T_2}^{-1}).
\end{align*}
\end{cor}

\subsection{Applying the general machinery} 
We show some arithmetic applications by invoking the general Euler system machinery of Jetchev--Nekov\'a{\v r}--Skinner  \cite{JNS}, see some details for the imaginary quadratic case in \cite[\S{4.3}]{Do-PhD} and \cite[\S{8}]{ACR}. These results will be used to deduce the Bloch--Kato conjecture and the anticyclotomic Iwasawa main conjecture by exploiting the relation between our Euler system classes and special values of complex and $p$-adic $L$-functions via an explicit reciprocity law.

For every ideal $\mu_3\in \mathcal{N}$, denote by
\[
z_{f,\chi,\mu_3}\in{\rm Sel}_{\relstr,\ord}(K_0[\mu_3],T_{f,\chi})
\]
the image of $\mathbf{z}_{f,\chi,\mu_3}$ from Theorem\,\ref{maintheorem2} under the projection
\begin{equation}\label{eq:Sel_projection_from_K_infty}
    {\rm Sel}_{\relstr,\ord}(K_0[\mu_3p^\infty],T_{f,\chi})\rightarrow{\rm Sel}_{\relstr,\ord}(K_0[\mu_3],T_{f,\chi}).
\end{equation}
And denote the base class
\[
z_{f,\chi}:={\rm Norm}^{K_0[1]}_{K_0}(z_{f,\chi,1})\in{\rm Sel}_{\relstr,\ord}(K_0,T_{f,\chi}).
\]
Note that since we assume $p \nmid h_{K_0}$, $K_0[1]$ is actually the same with $K_0$ (recall that $K_0[\fkn]$ is the maximal $p$-extension inside the ring class field of $K_0$ of conductor $\fkn$). Therefore, $z_{f,\chi}=z_{f,\chi,1}$.

\begin{thm}\label{thm:rank-1-general}
Assume that $f$ is not of CM-type, non-Eisenstein at $\mathfrak{P}$, and that $p\nmid 6h_{K_0}$. One has: 
$$z_{f,\chi} \text{ is non-torsion } \quad\Rightarrow\quad{\rm Sel}_{\relstr,\ord}(K_0,V_{f,\chi}) \text{ is one-dimensional.}$$

\end{thm}

\begin{proof}
 Combining Theorem~\ref{maintheorem2} and Proposition~\ref{propselmer2}, the system of classes 
\begin{equation}\label{eq:ES-rel-str}
\bigl\{z_{f,\chi,\mu_3}\in{\rm Sel}_{\relstr,\ord}(K_0[\mu_3],T_{f,\chi})\;\colon\;\mu_3\in\mathcal{N}\bigr\}
\end{equation}
forms an anticyclotomic Euler system in the sense of  Jetchev--Nekov\'a{\v r}--Skinner \cite{JNS} for the relaxed-strict-ordinary-ordinary Greenberg Selmer group. 

Under the assumption that $f$ is not of CM-type, the following properties (i)--(iii) follow from Momose's big image results \cite{Momo} as in \cite[Prop.\,7.1.4]{LLZ-K}:
\begin{itemize}
    \item[(i)] $V_{f,\chi}$ is absolutely irreducible;
    \item[(ii)] There is an element $\sigma\in G_{K_0}$ fixing $K_0[1]K_0(\mu_{p^\infty},(\cO_{K_0}^\times)^{1/p^\infty})$ such that $V_{f,\chi}/(\sigma-1)V_{f,\chi}$ is one-dimensional;
    \item[(iii)] There is an element $\gamma\in G_{K_0}$ fixing $K_0[1]K_0(\mu_{p^\infty},(\cO_{K_0}^\times)^{1/p^\infty})$ such that $V_{f,\chi}^{\gamma=1}=0$.
\end{itemize}
Hence, the fact that $z_{f,\chi}$ is non-torsion implies the one-dimensionality of ${\rm Sel}_{\relstr,\ord}(K_0,V_{f,\chi})$ by the general machinery of \cite{JNS}.
\end{proof}

Recall that $K_{0,\infty}^-$ is the anticyclotomic $\Z_p^2$ extension over $K_0$ and $\Lambda_{K_0}^-=\Z_p\dBr{{\rm Gal}(K_{0,\infty}^-/K_0)}$. 
Let $\mathbf{z}_{f,\chi,1}$ be the $\Lambda_{K_0}^-$-adic class of Theorem\,\ref{maintheorem2} of conductor $\mu_3=(1)$, and put the Iwasawa-theoretic base class
\[
\mathbf{z}_{f,\chi}:={\rm Norm}^{K_0[1]}_{K_0}(\mathbf{z}_{f,\chi,(1)}).
\]
Again $\mathbf{z}_{f,\chi}=\mathbf{z}_{f,\chi,(1)}$ from the assumption on the class number of $K_0$. Note that by Proposition \ref{propselmer2}, one has
\[
\mathbf{z}_{f,\chi}\in{\rm Sel}_{\relstr,\ord}(K_{0,\infty}^-,T_{f,\chi}).
\]

\begin{defn}\label{notation-bigimage}We say that $f$ {has big image at $\mathfrak{P}$} if the image of $G_\Q$ in ${\rm Aut}_{\cO}(T_f^\vee)$ contains a conjugate of ${\rm SL}_2(\Z_p)$.
\end{defn}

\begin{rem}By a theorem of Ribet \cite{Ribet-glasgow}, if $f$ is not of CM-type then it has big image for all but finitely many primes of $L$.
\end{rem}

Denote by 
\[
X_{\strrel,\ord}(K_{0,\infty}^-,A_{f,\chi})={\rm Hom}_{\Z_p}\bigl(\varinjlim{\rm Sel}_{\strrel,\ord}(K_0[\mathfrak{p}_3^r\bar{\mathfrak{p}}_3^s],A_{f,\chi}),\Q_p/\Z_p\bigr).
\]

One then has a divisibility towards an anticyclotomic Iwasawa main conjecture `without $p$-adic $L$-functions' as follows:
\begin{thm}\label{thm:IMC-general}
Assume that $f$ is not of CM-type, has big image at $\mathfrak{P}$, and that $p\nmid 6h_{K_0}$. 
If $\mathbf{z}_{f,\chi}$ is non-torsion, then: 
\begin{enumerate}
    \item $X_{\strrel,\ord}(K_{0,\infty}^-,A_{f,\chi})$ and ${\rm Sel}_{\relstr,\ord}(K_{0,\infty}^-,T_{f,\chi})$ both have $\Lambda_{K_0}^-$-rank one.
\item And we have the divisibility
\[
{\rm char}_{\Lambda_{K_0}^-}(X_{\strrel,\ord}(K_{0,\infty}^-,A_{f,\chi})_{\rm tors})\supset{\rm char}_{\Lambda_{K_0}^-}\biggl(\frac{{\rm Sel}_{\relstr,\ord}(K_{0,\infty}^-,T_{f,\chi})}{\Lambda_{K_0}^-\cdot\mathbf{z}_{f,\chi}}\biggr)^2
\]
in $\Lambda_{K_0}^-$. 
\end{enumerate}
Here, the subscript ${\rm tors}$ denotes the $\Lambda_{K_0}^-$-torsion submodule.
\end{thm}

\begin{proof}
Combining Theorem\,\ref{maintheorem2} and Proposition\,\ref{propselmer2}, the system of classes
\begin{equation}\label{eq:ES-Lambda}
\bigl\{\mathbf{z}_{f,\chi,\mu_3}\in{\rm Sel}_{\relstr,\ord}(K_0[\mu_3p^\infty],T_{f,\chi})\,\colon\;\mu_3\in\mathcal{N}\bigr\}
\end{equation}
forms a $\Lambda_{K_0}^-$-adic anticyclotomic Euler system in the sense of Jetchev--Nekov\'a{\v r}--Skinner for the relaxed-strict-ordinary-ordinary Selmer group.

Under the assumption that $f$ has big image at $\mathfrak{P}$, the following properties hold (see \cite[Prop.\,7.1.6]{LLZ-K})
\begin{itemize}
    \item[(i)] $\bar{T}_{f,\chi}:=T_{f,\chi}/\mathfrak{P}T_{f,\chi}$ is absolutely irreducible;
    \item[(ii)] There is an element $\sigma\in G_K$ fixing $K_0[1]K_0(\mu_{p^\infty},(\cO_K^\times)^{1/p^\infty})$ such that $T_{f,\chi}/(\sigma-1)T_{f,\chi}$ is free of rank $1$ over $\cO$;
    \item[(iii)] There is an element $\gamma\in G_{K_0}$ fixing $K_0[1]K_0(\mu_{p^\infty},(\cO_K^\times)^{1/p^\infty})$ and acting as multiplication by a scalar $a_\gamma\neq 1$ on $\bar{T}_{f,\chi}$;
\end{itemize}
and so the non-torsionness of $\mathbf{z}_{f,\chi}$ implies the conclusions by the general machinery of \cite{JNS}.
\end{proof}

\subsection{On the Bloch--Kato conjecture in rank $0$}
\label{subsec:BK-def}

Our first application is the Bloch--Kato conjecture in analytic rank zero for the conjugate self-dual $G_{K_0}$-representation $V_{f,\chi}=V_f^\vee(1-r)\otimes\chi^{-1}$.

\begin{assumption}\label{eq:chi-decomp}
    We assume that the anticyclotomic Hecke character $\chi$ over $K_0$ can be decomposed as: 
    \begin{equation*}
\chi=\tilde{\psi}_1\tilde{\psi}_2\mathbf{N}^{(k_1+k_2-2)/2},
    \end{equation*}
    where
    \begin{enumerate}
     \item $\psi_1$ is a Hecke character of $K_1$ of infinity type $(1-k_1,0)$, with $k_1\geq 1$, and modulus  $\fkf_1$.
    \item  $\psi_2$ is a Hecke character of $K_2$ of infinity type $(1-k_2,0)$, with $k_2\geq 1$ and modulus  $\fkf_2$.
    \item $\tilde{\psi}_i$ is the Hecke character of $K_0$, obtained by composing $\mathbb{A}_{K_0}^{\times}\xrightarrow{\mathbb{N}_{K_0/K_i}}\mathbb{A}_{K_i}^{\times}\xrightarrow{\psi_i}\mathbb{C}$ for each $i\in\{1,2\}$.
    \item By swapping $K_1$ and $K_2$, we may assume that ${k_2\ge k_1}$.
    \end{enumerate}
In this scenario, the infinity type of $\chi$ (corresponding to the order $(\cP_1,\cP_2,\cP_3,\cP_4)$ or $(1,\tau_3,\tau_2,\tau_1)$) is
\[\left(\frac{2-k_1-k_2}{2},\frac{k_1+k_2-2}{2},\frac{k_1-k_2}{2},\frac{k_2-k_1}{2}\right).\] 
\end{assumption}

\begin{thm}\label{thm:BK-def}
Let $f\in S_k(\Gamma_0(pN_f))$ be a $p$-ordinary $p$-stabilised newform of weight $k=2r\geq 2$ which is old at $p$. Let $\chi$ be an anticyclotomic Hecke character of $K_0$ as in (\ref{eq:chi-decomp}). Assume that:
\begin{enumerate}
    \item Either $k\ge k_1+k_2$ or $k_2-k_1\ge k$; 
    \item $N_f\cO_{K_3}=\mathfrak{n}^+\mathfrak{n}^-$ where $\fkn^+$
(respectively $\fkn^-$) is divisible only by primes which are split (respectively inert) in $K_0/K_3$ and $\fkn^-$ is a squarefree product of an even number of primes.
    %\item $f$ is not of CM type;
    \item $\bar{\rho}_f$ is absolutely irreducible; 
    \item $(pN_f,\mathrm{Norm}_{K_1/\Q}(\mathfrak{f}_1)\mathrm{Norm}_{K_2/\Q}(\mathfrak{f}_2)D_{K_0})=1$;
    \item $p\nmid 6h_{K_0}$, the class number of $K_0$;
\end{enumerate}
then we have the following implication
\[
L(f/K_0,\chi,r)\neq 0\quad\Longrightarrow\quad{\rm Sel}_{\rm BK}(K_0,V_{f,\chi})=0.
\] 
In other words, the Bloch--Kato conjecture holds in analytic rank zero for $V_{f,\chi}$.
\end{thm}

\begin{proof}
We consider the CM Hida families
\[
\bfg=\boldsymbol{\theta}_{\ch_1}(S_1),\quad\bfh=\boldsymbol{\theta}_{\ch_2}(S_2),
%\quad\textrm{where $(\ch_1,\ch_2)=(\beta\alpha,\beta^{-1}\alpha^{-\cc})$.}
\]
that pass through $\theta_{\psi_1}$ and $\theta_{\psi_2}$ respectively. Note that the triple $(\bff,\bfg_1,\bfg_2)$ satisfies (\ref{eq:+1}). Then the isomorphism (\ref{eq:dec-V-first}) of the associated $\Vsdag$ together with the specialization ${\mathcal{Q}_1}$ corresponding to $\theta_{\psi_1}$ and $\theta_{\psi_2}$ show that
\[
L(\mathbb{V}_{\mathcal{Q}_1}^\dag,0)=L(f/K_0,\chi,r).
\]

By Theorem\,\ref{thm:ERL} we then have
\[
L(f/K_0,\chi,r)\neq 0\quad\Longrightarrow\quad{\rm res}_p(\kappa(\bff,\bfg,\bfh))_{\bff,{\mathcal{Q}_1}}\neq 0.
\]
From our construction and Proposition \ref{prop:factor-S}, the class $\kappa(\bff,\bfg,\bfh)_{\mathcal{Q}_1}\in{\rm Sel}_{\relstr,\ord}(K_0,V_{f,\chi})$ is the base class of the anticyclotomic Euler system 
\[
\bigl\{z_{f,\chi,\mu_3}\in{\rm Sel}_{\relstr,\ord}(K_0[m],T_{f,\chi})\;\colon\;\mu_3\in\mathcal{N}\bigr\}
\]
of (\ref{eq:ES-rel-str}). Recall again that $K_0[1]=K_0$ from the assumption $p \nmid h_{K_0}$. By Theorem\,\ref{thm:rank-1-general}, we conclude that the Selmer group ${\rm Sel}_{\relstr,\ord}(K_0,V_{f,\chi})$ is one-dimensional, 
spanned by 
\[
z_{f,\chi}={\rm Norm}^{K_0[1]}_{K_0}(z_{f,\chi,1})=\kappa(\bff,\bfg,\bfh)_{\mathcal{Q}_1}.
\]

If $k\ge k_1+k_2$, we observe that the composition of maps in \ref{eq:map-ERL} corresponds to the composition of
\begin{align*}
    {\rm Sel}_{\relstr,\ord}(K_0,T_f^{\vee}(1-r)\otimes\tilde{\ch}_1^{-1}\tilde{\Psi}_{T_1}\tilde{\ch}_2^{-1}\tilde{\Psi}_{T_2}) \xrightarrow{{\rm res}_{\mathcal{P}_1}} &\rH^1(K_{0,\mathcal{P}_1},T_f^{\vee}(1-r)\otimes\tilde{\ch}_1^{-1}\tilde{\Psi}_{T_1}\tilde{\ch}_2^{-1}\tilde{\Psi}_{T_2})\\
    \longrightarrow &\rH^1(K_{0,\mathcal{P}_1},T_f^{\vee,-}(1-r)\otimes\tilde{\ch}_1^{-1}\tilde{\Psi}_{T_1}\tilde{\ch}_2^{-1}\tilde{\Psi}_{T_2})
\end{align*}
by equation (\ref{eq:shapiro}). Hence ${\rm res}_{\mathcal{P}_1}(z_{f,\chi})\neq 0$ by the reciprocity law Theorem \ref{thm:ERL}. The vanishing of ${\rm Sel}_{\ord,\ord}(K_0,V_{f,\chi})$ then follows by a standard argument using Poitou--Tate duality (see \cite[\S{5.1.1}]{Do-PhD}). This yields the result by using the Lemma~\ref{lem:BK-Gr} for $k\ge k_1+k_2$ to identify the latter group with ${\rm Sel}_{\rm BK}(K_0,V_{f,\chi})$.

If $k_2-k_1\ge k$, similarly by using ${\rm res}_{\mathcal{P}_2}(z_{f,\chi})\neq 0$ we obtain the vanishing of ${\rm Sel}_{\relstr,\relstr}(K_0,V_{f,\chi})$, which is again the Bloch-Kato Selmer group ${\rm Sel}_{\rm BK}(K_0,V_{f,\chi})$ by Lemma~\ref{lem:BK-Gr} for $k_2-k_1\ge k$.
\end{proof}
\begin{rem}\label{rootnumber}Let $\epsilon(f,\chi)$ to be the sign of the functional equation for $V_{f,\chi}$. Then $\epsilon(f,\chi)=\prod \epsilon(\pi_{K_{0,v}}\otimes \chi_v,1/2)$ over places $v$ of $K_0$ as a product of local root numbers. If $v|\fkn^+$ then $\epsilon(\pi_{K_{0,v}}\otimes \chi_v,1/2)=+1$ and if $v|\fkn^-$ then $\epsilon(\pi_{K_{0,v}}\otimes \chi_v,1/2)=-1$. Therefore the contribution from the local places is $+1$ due to assumption $(2)$. At the infinity places,
\begin{align*}
    \epsilon_{\infty}(\pi_{K_0}\otimes\chi,\frac{1}{2})&= i^{|k-1+(k_1+k_2-2)|+|k-1-(k_1+k_2-2)|+|k-1+(k_2-k_1)|+|k-1-(k_2-k_1)|} \\&=\begin{cases}
        +1   &\text{ if }\quad k>(k_1+k_2-2)\\
        -1 &\text{ if }\quad k_2-k_1<k\le k_1+k_2-2\\
        +1 &\text{ if } \quad k\le k_2-k_1.
    \end{cases}
\end{align*}
Hence conditions $(1)$ and $(2)$ of Theorem \ref{thm:BK-def} then imply that $\epsilon(f,\chi)=1$.
\end{rem}

\subsection{On the Iwasawa main conjecture}
\label{subsec:IMC-def}

Our second application is an evidence towards the anticyclotomic Iwasawa main conjecture for modular forms. Recall that we have an eigenform $f$ of weight $k=2r\geq 2$ with trivial nebentypus and an anticyclotomic character $\chi$ satisfying Assumption \ref{eq:chi-decomp}. Let
\[
A_{f,\chi}={\rm Hom}_{\Z_p}(T_f^\vee(1-r)\otimes\chi^{-1},\mu_{p^\infty}).
\]

\begin{thm}\label{thm:IMC-def}
Under the same assumption as in Theorem~\ref{thm:BK-def}, we assume further that:
\begin{enumerate}
    \item $\bar{\rho}_f$ is $p$-distinguished,
    \item  $f$ has big image,
    \item $p>k-2.$
\end{enumerate}
If $k\ge k_1+k_2$ then ${\rm Sel}_{\ord,\ord}(K_{0,\infty}^-,A_{f,\chi})$ is cotorsion over $\Lambda_{K_0}^-$. Furthermore, inside $\Lambda_{K_0}^-\otimes_{\Z_p}\Q_p$, we have the following inclusion
\[
{\rm char}_{\Lambda_{K_0}^-}\bigl({\rm Sel}_{\ord,\ord}(K_{0,\infty}^-,A_{f,\chi})^\vee\bigr)\supset\bigl(\mathscr{L}_p^{\unb,\eta_\bff}(f,\bfg_1,\bfg_2)^2\bigr).
\]

\end{thm}

\begin{proof}
Recall from Corollary \ref{cor:factor-S} that we have
\begin{equation}\label{eq:factor-S-sp}
{\rm Sel}^\bff(\Q,\mathbb{A}^\dagger)\simeq{\rm Sel}_{\ord,\ord}(K_{0,\infty}^-,A_f(r)\otimes\chi),
\end{equation}
where $\mathbb{A}^\dagger={\rm Hom}_{\Z_p}(\mathbb{V}^\dagger,\mu_{p^\infty})$. 

Note that from (\ref{eq:triple-2var}), 
$\mathscr{L}_p^{\unb,\eta_\bff}(f,\bfg_1,\bfg_2)$ is an element of $\cO\dBr{S_1,S_2}$. We then identify $\Lambda_{K_0}^-\simeq\cO\dBr{S_1,S_2}$ via the diagram (\ref{eq:G_K0^-identification}). The $p$-adic $L$-function $\mathscr{L}_p^{\unb,\eta_\bff}(f,\bfg_1,\bfg_2)$ is nonzero by \cite[Thm.\,C]{hung}. Note that our assumption that $k\ge k_1+k_2$ ensures that we are in the critical specializations i.e. $-k_{\sigma}/2< m_\sigma < k_{\sigma}/2$ for all $\sigma\in \Sigma$, following the notation of op. cit.. Hence  Theorem~\ref{thm:ERL} and the proof of Theorem \ref{thm:BK-def} implies that the class
\begin{equation}
\kappa(f,\bfg,\bfh)\in{\rm Sel}_{\relstr,\ord}(K_0,T_f^{\vee}(1-r)\otimes\tilde{\ch}_1^{-1}\tilde{\Psi}_{T_1}\tilde{\ch}_2^{-1}\tilde{\Psi}_{T_2})\nonumber
\end{equation}
is non-torsion. By construction, we can treat $\kappa(f,\bfg,\bfh)$ as the base class of the $\Lambda_{K_0}^-$-adic anticyclotomic Euler system 
\[
\bigl\{\mathbf{z}_{f,\chi,\mu_3}\in{\rm Sel}_{\relstr,\ord}(K_0[\mu_3p^\infty],T_{f,\chi})\,\colon\;\mu_3\in\mathcal{N}\bigr\}
\]
in (\ref{eq:ES-Lambda}). Via the isomorphism \ref{eq:shapiro}, the result follows immediately from
Theorem\,\ref{thm:IMC-general} applied to
\begin{equation}\label{eq:Lambda-class}
\mathbf{z}_{f,\chi}:={\rm Norm}^{K_0[1]}_{K_0}(\mathbf{z}_{f,\chi,(1)})=\kappa(f,\bfg,\bfh),
\end{equation}
the equivalence between two different formulation of the Iwasawa main conjecture in Proposition~\ref{prop:equiv}, and the Selmer group isomorphism  (\ref{eq:factor-S-sp}). 
\end{proof}

\begin{rmk}
Within Theorem \ref{thm:IMC-def}, the RHS can be compared to the $p$-adic $L$-function of Wan \cite[Thm.\,86]{wan_hilbertIMC} and Hung \cite{hung} (under assumptions \cite[Thm.\,11.1,11.2]{Fujiwara} and \cite[Thm.\,103]{wan_hilbertIMC}). The author then expects that the full Iwasawa Main Conjecture, which means an equality of Theorem \ref{thm:IMC-def}, will follow by combining Theorem \ref{thm:IMC-def} with the opposite inclusion of Wan \cite{wan_hilbertIMC} and the vanishing of the
$\mu$-invariant of the $p$-adic $L$-function \cite{hung} (those are generalizations of Skinner-Urban \cite{SU} and Vatsal \cite{vatsal-special}).
\end{rmk}

\begin{rmk}One expects a similar result that if $k_2-k_1\ge k$ then ${\rm Sel}_{\strrel,\strrel}(K_\infty^-,A_{f,\chi})$ is cotorsion over $\Lambda_{K_0}^-$ together with the following inclusion inside $\Lambda_{K_0}^-\otimes_{\Z_p}\Q_p$,
\[
{\rm char}_{\Lambda_{K_0}^-}\bigl({\rm Sel}_{\strrel,\strrel}(K_{0,\infty}^-,A_{f,\chi})^\vee\bigr)\supset\bigl(\mathscr{L}_p^{\bfh,\eta_\bfh}(f,\bfg_1,\bfg_2)^2\bigr).
\]
The only missing ingredient is the non-vanishing of the $p$-adic $L$-function in this region.
    
\end{rmk}    
\subsection{On the Bloch--Kato conjecture in rank $1$} 

Our last application is extracted from the proof of Theorem~\ref{thm:IMC-def}. It provides a result towards the Bloch--Kato conjecture in rank $1$.

\begin{thm}\label{thm:BK-def-1}
Under the same assumption as in Theorem~\ref{thm:IMC-def}, if $k_1+k_2-2\geq k\ge k_2-k_1+2$ (which induces $L(f/K,\chi,r)=0$), then 
\[
{\rm dim}_{L_\mathfrak{P}}\,{\rm Sel}_{\rm BK}(K_0,V_{f,\chi})\geq 1.
\]

\end{thm}

\begin{proof}
 The class $\mathbf{z}_{f,\chi}\in{\rm Sel}_{\relstr,\ord}(K_{0,\infty}^-,T_{f,\chi})$ is non-torsion via the proof of Theorem~\ref{thm:IMC-def}. Furthermore, $\mathbf{z}_{f,\chi}$ is the base of a $\Lambda_{K_0}^-$-adic anticyclotomic Euler system as in (\ref{eq:Lambda-class}) for the relaxed-strict-ordinary-ordinary Selmer group. Theorem~\ref{thm:IMC-general} then implies that ${\rm Sel}_{\relstr,\ord}(K_{0,\infty}^-,T_{f,\chi})$ has $\Lambda_{K_0}^-$-rank $1$. The natural map (compare with the projection (\ref{eq:Sel_projection_from_K_infty}))
\begin{equation}\label{eq:proj}
{\rm Sel}_{\relstr,\ord}(K_{0,\infty}^-,T_{f,\chi})/(\gamma_{1,-}-1,\gamma_{2,-}-1)\rightarrow{\rm Sel}_{\relstr,\ord}(K_0,T_{f,\chi})
\end{equation}
is injective (see also \cite[Prop.\,5.3.14]{MR-KS} and \cite[p.\,72]{greenberg-cetraro}). 

Hence, the Selmer group ${\rm Sel}_{\relstr,\ord}(K_0,T_{f,\chi})$ has a positive $\cO$-rank. The theorem then follows by Lemma \ref{lem:BK-Gr}, which computes the local conditions of the Bloch-Kato Selmer group explictly. 
\end{proof}

\begin{rmk}
Note that by letting $z_{f,\chi}\in{\rm Sel}_{\relstr,\ord}(K_0,T_{f,\chi})$ be the image of $\mathbf{z}_{f,\chi}$ under the map (\ref{eq:proj}), such a class $z_{f,\chi}\in{\rm Sel}_{\rm BK}(K_0,V_{f,\chi})$ satisfies:
\[
z_{f,\chi}\neq 0\quad\Longrightarrow\quad{\rm dim}_{L_\mathfrak{P}}\,{\rm Sel}_{\rm BK}(K_0,V_{f,\chi})=1.
\]
by Theorem~\ref{thm:rank-1-general}.
\end{rmk}

\newpage
\bibliographystyle{amsalpha}
\bibliography{Schoen-cm}

\end{document}